\documentclass{amsart}

\usepackage{upref,amsxtra,amscd,amsopn,amsbsy,amstext,amssymb,amsmath,amsthm}
\usepackage{bm}
\usepackage{amsfonts}
\usepackage{mathrsfs}
\usepackage{braket}
\usepackage{color}
\usepackage{graphicx}
\usepackage{hyperref}

\makeatletter
\@addtoreset{equation}{section}

\makeatother

\newtheorem{theorem}{Theorem}[section]
\newtheorem{proposition}[theorem]{Proposition}
\newtheorem{corollary}[theorem]{Corollary}
\newtheorem{lemma}[theorem]{Lemma}
\theoremstyle{definition}

\newtheorem{remark}[theorem]{Remark}

\allowdisplaybreaks

\begin{document}

\title[Willmore obstacle problems]{Willmore obstacle problems \\ under Dirichlet boundary conditions}

\author[H.-Ch.~Grunau]{Hans-Christoph Grunau}
\address[H.-Ch.~Grunau]{Fakult\"{a}t f\"{u}r Mathematik, Otto-von-Guericke-Universit\"{a}t, Postfach 4120, \\39016 Magdeburg, Germany}
\email{hans-christoph.grunau@ovgu.de}
\author[S.~Okabe]{Shinya Okabe}
\address[S.~Okabe]{Mathematical Institute, Tohoku University, Aoba, Sendai 980-8578, Japan}
\email{shinya.okabe@tohoku.ac.jp}

\keywords{Obstacle problem, elastic energy, Willmore energy, surfaces of revolution, Dirichlet boundary conditions}
\subjclass[2010]{Primary: 49Q10, 49J40, Secondary: 49J05, 49J10, 53C42, 35J87}

\date{\today}

\begin{abstract}
We consider obstacle problems  under Dirichlet boundary conditions
for Euler's elastica functional in the class of one-dimensional graphs over the real axis and 
for the  Willmore functional in the class of  surfaces of revolution. 
We prove the existence of minimisers of the obstacle problems under the assumption that the elastic or the  Willmore energy resp. under the unilateral constraint is below a universal bound. We address the question whether such bounds are necessary in order to ensure the solvability of the obstacle problems. 
Moreover, we give several instructive examples of obstacles such that minimisers exist. 
\end{abstract}

\maketitle

\section{Introduction} \label{section:1}

Euler's elastica functional, which is the integral of the squared curvature along a curve, and its two-dimensional analogue, the integral of the squared mean curvature over a surface, were introduced already in the 18th and 19th century respectively, see \cite{Euler} and \cite{Poisson}. The latter is nowadays called after Willmore who reintroduced it in the 1960's, see \cite{Wil65}. Employing the modern theory of the Calculus of Variations a number of impressive results has been proved since then, mainly for closed surfaces. One may see e.g. the survey articles by Kuwert and Sch\"atzle \cite{KuSch} and by Marques and Neves \cite{MarquesNeves} and references therein.

Boundary value problems for the elastica and for the Willmore functional have gained attention only more recently. Dirichlet boundary conditions prescribe the position and the direction of the boundary of the unknown curve or surface,
while Douglas or Navier boudary conditions prescribe only the position. The latter results in a natural boundary condition for the (mean) curvature.
As for the elastica functional one may see e.g. \cite{DD}, which relies on Euler's observations, references therein and subsequent works.
An existence results for the Dirichlet problem for the Willmore functional was proved  by Sch\"atzle \cite{Schaetzle} in the  class of branched immersions in $\mathbb{R}^3\cup \{\infty\}$. 
In \cite{DGR} existence of minimisers of a relaxed Willmore functional in the class of graphs over two-dimensional domains was proved.
The Douglas (or Navier) boundary value problem was studied by Novaga and Pozzetta in \cite{NovagaPozzetta,Pozzetta}.
In the class of surfaces of revolution existence results for  Willmore minimisers under Dirichlet boundary conditions were obtained in \cite{DDG,DFGS,EichGr}; see also references therein. Imposing axial symmetry decreases the analytical difficulties and permits to uncover detailed analytical and geometric properties of  minimisers.

 For 
a special (Canham-) Helfrich functional, where an area term is added to the Willmore functional, an existence result for branched immersions in $\mathbb{R}^3$ was found by Eichmann in \cite{Eichmann2019}. This is somehow related to \cite{dLPR}, where 
Da~Lio,  Palmurella, and  Rivi\`{e}re imposed an area constraint  in order to minimise the Willmore functional. 

The present paper is concerned with \emph{obstacle problems} under
Dirichlet boundary conditions for Euler's elastica functional and  the  Willmore functional.  More precisely, we consider minimisation problems for these functionals among curves or surfaces of revolution resp. with Dirichlet boundary conditions under a  unilateral constraint. 
Recently the minimisation problem for  the  elastica 
(or the one-dimensional Willmore) functional 
among graphs of functions over the real axis with Navier boundary conditions 
under a unilateral constraint has been intensively studied (\cite{DD,Miura,Mueller_2019,Yoshizawa}): 
For a given obstacle function $\psi : [0,1] \to \mathbb{R}$, find a function $u : [0,1] \to \mathbb{R}$ such that $u$ attains 
\begin{equation}
\label{eq:1.1}
\inf_{v \in K(\psi)} W(v), 
\end{equation}  
where
\begin{align*}
W(v) &:= \int^1_0 \kappa(x)^2 \sqrt{1+v'(x)^2}\, dx 
         = \int^1_0 \left( \dfrac{v''(x)}{(1+v'(x)^2)^{3/2}} \right)^2 \sqrt{1+v'(x)^2}\, dx \\
         &= \int^1_0 \dfrac{v''(x)^2}{(1+v'(x)^2)^{5/2}}\, dx, 
\end{align*}
is Euler's  elastica (or  the one-dimensional Willmore) functional and 
\begin{equation*}
K(\psi):= \{ v  \in H^2(0,1) \mid v(0)=v(1)=0, \,\, v \ge \psi \,\,\, \text{in} \,\,\, [0,1] \}
\end{equation*}
is the class of admissible functions.
One should have in mind that it is important to work in classes of graphs
(which is somehow a further obstacle condition) because otherwise, thanks to the scaling behaviour of the elastica functional, one may have sequences of arbitrarily large nonprojectable curves with ``very small'' elastic energy,
cf. Figure 2 and the corresponding remarks in the introduction of \cite{Mueller_2019}.
Problem \eqref{eq:1.1} was firstly studied by Dall'Acqua and Deckelnick \cite{DD} and 
they proved that $W$ has a minimiser in $K(\psi)$ under an explicit smallness condition on $\psi$. 
In view of this result it is an obvious question whether there is a threshold for 
the obstacles beyond which they no longer permit a solution of the minimisation problem for $W$ in $K(\psi)$. 
For symmetric cone obstacles $\psi$, i.e., $\psi(x)=\psi(1-x)$ for all $x \in [0,1]$ and $\psi$ is affine on $(0, 1/2)$, such that $\psi(0)=\psi(1)<0$ and $\psi(1/2)>0$, 
the question was completely solved as follows: 
(i) if $\psi(1/2) < 2/c_0$, then there exists a unique minimiser of $W$ in $K_{sym}(\psi)$; (ii) if $\psi(1/2) \ge 2/c_0$, then there is no minimiser of $W$ not only in $K_{sym}(\psi)$ but also in $K(\psi)$, 
where 
\begin{equation*}
K_{sym}(\psi) := \{ v \in K(\psi) \mid v(x)=v(1-x) \,\,\, \text{for all} \,\,\, x \in [0,1] \},  
\end{equation*}
and 
\begin{equation*}
c_0 := \int_{\mathbb{R}} \dfrac{d \tau}{(1 + \tau^{2})^{5/4}}
= \sqrt{\pi} \frac{\Gamma(3/4)}{\Gamma(5/4)}=2.396280469\ldots. 
\end{equation*}
We note that the existence of a  minimiser in the first case follows from  Dall'Acqua--Deckelnick~\cite{DD}. Its uniqueness 
was independently proved by Miura~\cite{Miura} and Yo\-shi\-za\-wa~\cite{Yoshizawa}.
The non-existence of minimisers in the case $\psi(1/2)>2/c_0$ was proved by M\"uller~\cite{Mueller_2019} and in the critical case $\psi(1/2)=2/c_0$ independently  by Miura~\cite{Miura} and   Yoshizawa~\cite{Yoshizawa}.
  
One of the purposes of this paper is to extend the studies of problem \eqref{eq:1.1} to one-dimensional graphs of functions and to surfaces of revolution with \emph{Dirichlet} boundary conditions. 
In the one-dimensional setting, we impose the following boundary conditions
\begin{equation}\label{eq:1.0}
u(0)=u(1)=0 \quad 
\mbox{\ as well as\ }\quad 
u'(0)=u'(1)=0.
\end{equation}
The conditions on the derivative will imply that in contrast to \cite{DD}
minimisers (and even admissible functions) may no longer be concave and look
close to $0$ and $1$ really different from those in \cite{DD}.
Special emphasis is laid on studying the necessity of smallness conditions for  general symmetric obstacles.

We describe now our main results in some more detail.
First we consider the obstacle problem for $W(\cdot)$ among symmetric
one-dimensional  graphs. 
Namely, we only consider symmetric obstacles $\psi$ which are subject to the following basic condition: 
\begin{align} \label{eq:A} \tag{A}
& \psi \in C^0([0,1]),\quad 
\psi(x)= \psi(1-x) \,\,\,\text{for all} \,\,\, x \in (0,1),\\
& \psi (0)=\psi (1) <0,\quad 
\exists x_0\in (0,1): \psi (x_0) >0.\nonumber
\end{align}
The latter condition is imposed to avoid $u(x)\equiv 0$ as a possible trivial minimiser.
The negativity condition is needed to ensure regularity of minimisers up to the boundary $\{0,1\}$.
For such $\psi$ we consider the minimisation problem   
\begin{equation}
\label{eq:1.2}
\inf_{v \in \mathcal{M}(\psi)} W(v) 
\end{equation}  
with 
\begin{equation}\label{eq:1.2.a}
\mathcal{M}(\psi) := \{ u \in H^2_0(0,1) \mid u(x) = u(1-x), \,\,\, u(x) \ge \psi(x), \,\,\, \text{for all} \,\,\, x \in (0,1) \}. 
\end{equation}
Note that $u\in H^2_0(0,1) $ encodes the Dirichlet boundary conditions in \eqref{eq:1.0}.

As for the existence of minimisers we have: 
\begin{theorem} \label{theorem:1.1}
Assume that $\psi$ satisfies conditions \eqref{eq:A} and 
\begin{equation} \label{eq:B} \tag{B}
\inf_{v \in \mathcal{M}(\psi)} W(v) < 4c_0^2. 
\end{equation}
Then there exists $u \in \mathcal{M}(\psi)$ such that 
\begin{align*}
W(u) = \inf_{v \in \mathcal{M}(\psi)} W(v). 
\end{align*}
\end{theorem}
For the regularity of minimisers $u$ of $W$ in $\mathcal{M}(\psi)$, we obtain  as in~\cite{DD} that $u\in W^{3,\infty}$ and $u'''\in BV$, see Proposition \ref{theorem:2.8}. 
Moreover, Theorem~\ref{theorem:1.1} can be extended to the non-symmetric case (see Theorem~\ref{theorem:2.19}), but then a more restrictive smallness condition has to be imposed. 

In Remark \ref{theorem:2.4} we explain how typical obstacles look like such that condition~\eqref{eq:B} is satisfied. While we have to leave open whether condition~\eqref{eq:B} is optimal, 
Theorem~\ref{theorem:1.2} below shows that indeed some kind of smallness condition is necessary in order to have existence of minimisers of $W$ in $\mathcal{M}(\psi)$.
It is natural to ask whether there exists a specific universal bound such that problem~\eqref{eq:1.2} has no solution, if the obstacle violates this bound. 
For this question we obtain an affirmative answer: 
\begin{theorem} \label{theorem:1.2}
We assume that the obstacle $\psi$ satisfies condition~\eqref{eq:A} and that there exists a minimiser $u \in \mathcal{M}(\psi)$  of $W(\cdot)$. 
Then the obstacle has to obey the following bound$\colon$
\begin{equation}\label{eq:1.3}
\forall x\in [0,1]:\quad \psi (x) \le  \max_{y\in [0,\infty]} 
 \frac{	1+(1+y^2)^{-1/4}  }{	c_0-G(y)  }
 =1.1890464540\ldots.
\end{equation}
\end{theorem}
Here the function $G$, which appears in \eqref{eq:1.3} and plays a crucial role in the one-dimensional elastica equation, is defined by
\begin{align}\label{eq:1.1.new}
G: \mathbb{R} \to \left( -\dfrac{c_0}{2}, \dfrac{c_0}{2} \right), \quad 
G(t):= \int^{t}_{0} \dfrac{d \tau}{(1 + \tau^{2})^{5/4}}. 
\end{align}

Theorem~\ref{theorem:1.2} says that for any obstacle $\psi$ violating \eqref{eq:1.3} the minimisation problem has no solution in the class  $\mathcal{M}(\psi)$.

We derive \eqref{eq:1.3} as a universal bound for all sufficiently smooth supersolutions of the elastica equation
under Dirichlet boundary conditions. 
Although its optimality for ``admissible'' obstacles $\psi$ (which permit a minimiser) is not proved here, 
we are confident that the universal bound cannot be improved for sufficiently smooth supersolutions, see Remark~\ref{theorem:2.16}. 
Moreover, the proof of Theorem~\ref{theorem:1.2} can be adapted to the Navier boundary value problem, where one can obtain the universal bound $2/c_0$. 
This means that the proof of Theorem~\ref{theorem:1.2} yields a significant generalisation of~\cite[Lemma 4.3]{DD} and~\cite[Theorem 1.1]{Mueller_2019}. 

We adapt Theorem \ref{theorem:1.1} to the minimisation problem for the Willmore functional for surfaces of revolution 
under Dirichlet boundary conditions and a unilateral constraint. 
Let $u : [-1,1] \to (0, \infty)$ be a profile curve with Dirichlet boundary conditions 
\begin{equation}
\label{eq:1.4}
u(1)=u(-1)=\alpha>0, \quad u'(1)=u'(-1)=0. 
\end{equation}
Then the Willmore functional of the corresponding surface of revolution
${\mathcal R}(u): (x,\theta)\mapsto (x,u(x)\cos (\theta),u(x)\sin(\theta))$
with mean curvature $H$ (defined as the mean value of the principal curvatures) is given by 
\begin{align*}
W(u) &:= \int_{{\mathcal R}(u)}H^2 \, dS\\
 &= \dfrac{\pi}{2} \int^1_{-1} \left( \dfrac{1}{u(x) \sqrt{1+(u'(x))^2}} - \dfrac{u''(x)}{(1+(u'(x))^2)^{3/2}} \right)^2 u(x) \sqrt{1 + (u'(x))^2} \, dx.  
\end{align*}
Working in the class of projectable profile curves (i.e. graphs) is helpful in order to ensure compactness of suitable minimising sequences.
 
For obstacles $\psi $ satisfying 
\begin{equation}
\label{eq:C} \tag{C}
\psi \in C^0([-1,1]; (0, \infty)),\quad 
\psi(1)=\psi(-1)>\alpha, \quad \psi(x)=\psi(-x) \,\,\,\text{for all}\,\,\, x \in (-1,1), 
\end{equation}
we consider the minimisation problem 
\begin{equation}
\label{eq:1.5}
\inf_{v \in N_\alpha(\psi)} W(v),  
\end{equation}
with 
\begin{align}
N_\alpha &:=  \{ v \in H^2((-1,1);(0,\infty)) \mid \text{$v$ satisfies \eqref{eq:1.4}},\nonumber\\ 
         & \qquad \qquad \qquad \qquad \qquad \qquad \quad
         v(x)=v(-x) \,\, \text{for all} \,\, x \in [-1,1] \},\nonumber\\
N_\alpha(\psi) &:= \{ v \in N_\alpha \mid v(x) \le \psi(x)\,\, \text{for all} \,\, x \in [-1,1] \}. 
\label{eq:1.5.a}
\end{align}
This means that in contrast to the one-dimensional situation, we look for minimising functions \emph{below} the given obstacle.
Typically one should think of obstacles which satisfy $\psi(x) <\alpha$ 
for $|x|$ close to $0$. Typical ``admissible'' obstacles (where a minimiser exists) are constructed  in Remark~\ref{theorem:3.3}.

We think that this is here the more interesting setting because solutions of the obstacle problem look then completely different from minimisers of the free problem (without obstacle), which are all strictly above $\alpha$. The arguments for minimisation in classes of functions above a given obstacle, however, will be similar.

First we have similarly  as in Theorem \ref{theorem:1.1}: 
\begin{theorem} \label{theorem:1.3}
Assume that $\psi$ satisfies conditions~\eqref{eq:C} and 
\begin{equation} \label{eq:D} \tag{D}
\inf_{v \in N_\alpha(\psi)} W(v) < 4 \pi. 
\end{equation}
Then there exists $u \in N_\alpha(\psi)$ such that 
\begin{align*}
W(u)= \inf_{v \in N_\alpha(\psi)} W(v). 
\end{align*}
\end{theorem}
We obtain the same regularity result of minimisers $u$ of $W$ in $N_\alpha(\psi)$ as above: $u\in W^{3,\infty}$, $u'''\in BV$, see Proposition~\ref{theorem:3.10}. 

We show in Proposition~\ref{theorem:3.4} that for $\alpha \to \infty$, only almost
constant functions $v(x)\approx \alpha$ satisfy  condition \eqref{eq:D}. This means that
Theorem~\ref{theorem:1.3} is interesting (only) for ``small'' $\alpha$, while for $\alpha \to \infty$ we should seek a condition different from (and weaker than)  \eqref{eq:D}. 
Indeed, we have: 
\begin{theorem} \label{theorem:1.4}
Assume that $\psi$ satisfies condition~\eqref{eq:C} and is such that 
\begin{equation}\label{eq:1.6}
\inf_{v \in N_\alpha(\psi)} W(v) <\pi \max_{S\in [0,\alpha]}g_\alpha (S) \mbox{\ with\ }  g_\alpha (S):=(\alpha-S)G(S)^2 .
\end{equation}
Then there exists $u \in N_\alpha(\psi)$ such that 
\begin{align*}
W(u)= \inf_{v \in N_\alpha(\psi)} W(v). 
\end{align*}
\end{theorem}
We discuss in Remark~\ref{theorem:3.5} that condition~\eqref{eq:1.6} is actually weaker than condition~\eqref{eq:D} for large $\alpha$ (beyond $\approx 6.1$). The occurrence of the function $G$ indicates that in this case 
 one-dimensional arguments come into play again.

We construct many examples belonging to the admissible sets $\mathcal{M}(\psi)$ and $N_\alpha(\psi)$ by employing particular solutions  of the elastica equation
and prototype Willmore surfaces of revolution (spheres and catenoids). 
These examples illustrate how obstacles may look like in order to obey condition \eqref{eq:B}, \eqref{eq:D} or \eqref{eq:1.6}, respectively.

The gradient flow for $W$ defined on graphs with a unilateral constraint 
and satisfying homogeneous Navier boudary conditions has also recently been studied in \cite{Mueller_2020,Mueller_2021,Okabe-Yoshizawa,Yoshizawa}. 
It would be  also interesting to investigate  gradient flows corresponding to problems \eqref{eq:1.2} and \eqref{eq:1.5}.   

This paper is organised as follows: 
Section~\ref{section:2} is devoted to the obstacle problem for Euler's elatica (or the one-dimentional Willmore) functional with Dirichlet boundary conditions. 
In Section~\ref{subsection:2.1} we introduce some notations and collect fundamental facts about the elastica equation.  
We prove Theorem~\ref{theorem:1.1} in Section~\ref{subsection:2.2} and the regularity of
minimisers in Section~\ref{subsection:2.3}. 
We prove Theorem~\ref{theorem:1.2} in Section~\ref{subsection:2.4}. 
In Section~\ref{subsection:2.5} we briefly study the non-symmetric case. 

Section~\ref{section:3} is concerned with the obstacle problem for the Willmore functional defined on surfaces of revolution, again with Dirichlet boundary conditions. 
In Section~\ref{subsection:3.1} we recall the basic existence and symmetry result for
``free'' minimisers (with no obstacle). 
In Section~\ref{subsection:3.2} we prove Theorems~\ref{theorem:1.3} and~\ref{theorem:1.4}
and present many instructive examples. 
Finally we prove the regularity of minimisers in Section~\ref{subsection:3.3}. 
\section{The one-dimensional Dirichlet obstacle problem} \label{section:2}

\subsection{Explicit solutions for the one-dimensional elastica equation}\label{subsection:2.1}

We first collect some facts concerning the elatica (or the one-dimensional Willmore)  equation
without obstacle which were already known to Euler \cite[pp. 231--297]{Euler} (see in particular 
pp. 233--234) and which will be relevant also in order to understand the shape of 
admissible obstacles in what follows. For a more convenient reference one may
also see \cite{DG07}.

For $u \in H^2_0(0,1)$, i.e. a sufficiently smooth function subject to homogeneous 
Dirichlet boundary conditions (horizontal clamping),
we recall the one-dimensional  elastica functional: 
$
W(u) = \int^1_0 \kappa(x)^2 \sqrt{1+u'(x)^2}\, dx . 
$ 
Its critical points 
satisfy the elastica (or one-dimensional Willmore) equation:
\begin{align} \label{eq:2.1}
\dfrac{1}{\sqrt{1 + u'(x)^2}} \dfrac{d}{dx} \left( \dfrac{\kappa'(x)}{\sqrt{1 + u'(x)^2}} \right) + \dfrac{1}{2} \kappa(x)^3 = 0, \quad x \in (0,1). 
\end{align} 
For the graph of $u : [0,1] \to \mathbb{R}$, its curvature is given by 
\begin{align*}
\kappa(x):= \kappa_u(x):= \dfrac{u''(x)}{(1+u'(x)^2)^{3/2}}. 
\end{align*}

Symmetric solutions to equation \eqref{eq:2.1} are known explicitly. Here the function $G$ defined in 
\eqref{eq:1.1.new} plays an important role. This function is smooth and strictly increasing and so is 
$G^{-1}: (-c_0/2,c_0/2)\to \mathbb{R}$.

\begin{lemma}{\,\,{\rm (\cite[Lemma 4]{DG07})}} \label{theorem:2.1}
Let $u \in C^4([0,1])$ be a function symmetric around $x=1/2$. 
Then $u$ solves the elastica equation {\rm \eqref{eq:2.1}} iff there exists $c \in (-c_0, c_0)$ such that 
\begin{align*}
u'(x) = G^{-1}\left( \dfrac{c}{2} - cx \right) \quad \text{in} \quad [0, 1]. 
\end{align*}
For the curvature, one has that 
\begin{align*}
\kappa(x) = - \dfrac{c}{\sqrt[4]{1 + G^{-1}\left( \tfrac{c}{2} - c x \right)^2}}. 
\end{align*}
Moreover, if we additionally assume that $u(0)=u(1)=0$, then one has 
\begin{align*}
u(x)= \dfrac{2}{c \sqrt[4]{1 + G^{-1}\left( \tfrac{c}{2} - c x \right)^2}} - \dfrac{2}{c \sqrt[4]{1 + G^{-1}\left( \tfrac{c}{2}\right)^2}}, 
\quad c \neq 0. 
\end{align*}
Finally, if one additionally assumes also that $u'(0)=u'(1)=0$, then one has
\begin{align*}
 u(x)\equiv 0.
\end{align*}
\end{lemma}
One should observe that this result holds true \emph{only in the class of graphs of smooth functions}.
The graph of the function $\hat{u}_{c_0}$ as displayed in Figure~\ref{figure:1} is 
a nontrivial solution to the Dirichlet problem for the elastica equation
under homogeneous boundary conditions, but:
\begin{enumerate}
\item[$\bullet$] although the curve is smooth and a graph, due to  $\hat{u}'_{c_0}(1/4)=-\hat{u}'_{c_0}(3/4)=\infty$, 
       it is \emph{not a smooth graph solution},
\item[$\bullet$] it is not minimising the elastica functional.
\end{enumerate}

\subsection{The obstacle problem in the symmetric case} \label{subsection:2.2}

In this section we only consider symmetric obstacles $\psi$ satisfying condition \eqref{eq:A}. 
Thanks to this condition the  set $\mathcal{M}(\psi)$ of  admissible functions, which is defined in \eqref{eq:1.2.a}, is not empty. 
One should observe that unlike Section~\ref{section:3} we consider functions \emph{above} the given obstacle.
We assume further  that $\psi$ satisfies condition \eqref{eq:B}, i.e., 
\begin{equation*} 
\inf_{v \in \mathcal{M}(\psi)} W(v) < 4c_0^2. 
\end{equation*}

\begin{lemma} \label{theorem:2.2}
Assume that $v \in \mathcal{M}(\psi)$ satisfies $W(v) \le 4 c_1^2$ for some $c_1\in [0,c_0)$. 
Then 
\begin{align*}
\max_{x \in [0, 1]} |v'(x)| \le G^{-1}\left( \dfrac{c_1}{2} \right) < \infty. 
\end{align*}
\end{lemma}
\begin{proof}
Let $x_{\rm max} \in (0,1)$ be such that 
\begin{align*}
v'(x_{\rm max}) = \max_{x \in [0, 1]}| v'(x) |. 
\end{align*}
Since $v \in \mathcal{M}(\psi)$ is symmetric, we have $v'(x_{\rm max}) = - v'(1-x_{\rm max})$. 
To begin with, we consider the case where $x_{\rm max} \in (0, 1/2)$. 
By H\"older's inequality, we have 
\begin{align*}
& 4 G(v'(x_{\rm max})) \\
&= G(v'(x_{\rm max})) + \{ G(v'(x_{\rm max})) - G(v'(1-x_{\rm max}))\} - G(v'(1-x_{\rm max})) \\
&= \int^{v'(x_{\rm max})}_{0} \dfrac{d \tau}{(1+ \tau^2)^{\frac{5}{4}}} +  \int^{v'(x_{\rm max})}_{v'(1-x_{\rm max})} \dfrac{d \tau}{(1+ \tau^2)^{\frac{5}{4}}} 
       + \int^{0}_{v'(1-x_{\rm max})} \dfrac{d \tau}{(1+ \tau^2)^{\frac{5}{4}}} \\
&= \int^{x_{\rm max}}_{0} \dfrac{v''(x)}{(1 + v'(x)^2)^{\frac{5}{4}}} \, dx 
     + \int^{1-x_{\rm max}}_{x_{\rm max}} \dfrac{- v''(x)}{(1 + v'(x)^2)^{\frac{5}{4}}} \, dx 
     + \int^{1}_{1-x_{\rm max}} \dfrac{v''(x)}{(1 + v'(x)^2)^{\frac{5}{4}}} \, dx \\
& \le \int^1_0 |\kappa_v(x)| (1 + v'(x)^2)^{1/4} \, dx 
  \le W(v)^{1/2} \le 2 c_1,   
\end{align*}
i.e., 
\begin{align*}
G(v'(x_{\rm max})) \le \dfrac{c_1}{2}. 
\end{align*}
Then it follows from the monotonicity of $G$ that 
\begin{align} \label{eq:2.2}
v'(x_{\rm max}) \le G^{-1}\left( \dfrac{c_1}{2} \right) < \infty. 
\end{align}
We turn to the case where $x_{\rm max} \in (1/2, 1)$. 
Considering $-v$ instead of $v$ and $1-x_{\rm max}$  instead of $x_{\rm max}$, we observe  as in the first case that \eqref{eq:2.2} holds. 
The proof is complete. 
\end{proof}

The preceding estimate is the key to prove our first existence result. 

\begin{proof}[Proof of Theorem \ref{theorem:1.1}]
Let $\{ u_j \} \subset \mathcal{M}(\psi)$ be a minimising sequence, i.e., 
\begin{align*}
\lim_{j \to \infty} W(u_j) =  \inf_{v \in \mathcal{M}(\psi)} W(v). 
\end{align*} 
Thanks to condition \eqref{eq:B} we find a constant 
 $c_1 \in [0, c_0)$ and a subsequence so that we have  
\begin{align} \label{eq:2.3}
W(u_j) \le 4 c_1^2 \quad \text{for any} \quad j \in \mathbb{N},  
\end{align}
where for brevity $\{u_j \}$ denotes also the subsequence. 
Then, thanks to Lemma \ref{theorem:2.2}, we obtain 
\begin{align*}
|u_j(x)| + |u'_j(x)| \le \int^{\tfrac{1}{2}}_0 |u'_j(\xi)|\, d\xi + |u'_j(x)| 
 \le \dfrac{3}{2} G^{-1}\left( \dfrac{c_1}{2} \right)
\end{align*}
for any $x \in [0, 1]$. Thus $\{ u_j \}$ is bounded in $C^1(0,1)$. 
This together with \eqref{eq:2.3} implies that $\{ u_j \}$ is also bounded in $H^2(0,1)$. 
Indeed, we have 
\begin{equation*}
4 c_1^2 \ge W(u_j) \ge \dfrac{1}{(1+ G^{-1}(c_1/2)^2)^{5/2}}\int^1_0 u''(x)^2 \, dx. 
\end{equation*}
Thus we find $u \in H^2_0(0,1)$ such that 
\begin{align} \label{eq:2.4}
u_j \rightharpoonup u \quad \text{weakly in} \quad H^2_0(0,1), 
\end{align}
up to a subsequence. Since the embedding $H^2_0(0,1) \subset C^{1, \gamma}(0,1)$ is compact for any $\gamma \in (0, 1/2)$, 
 we have in particular that
\begin{align} \label{eq:2.5}
u_j \to u \quad \text{in} \quad C^1(0,1). 
\end{align}
Then it follows from \eqref{eq:2.4} and \eqref{eq:2.5} that 
\begin{align*}
\liminf_{j \to \infty} W(u_j) \ge W(u). 
\end{align*}
Recalling that \eqref{eq:2.5} yields $u \in \mathcal{M}(\psi)$, we see that $u$ is a  minimiser of $W$. 
The proof is complete. 
\end{proof} 

\begin{remark}\label{theorem:2.3}
{\rm\ 
Since $\mathcal{M}(\psi)\subset H^2_0 (0,1)$ is convex, one finds as usual 
that any minimiser $u\in \mathcal{M}(\psi)$ of $W(\cdot)$ satisfies the variational 
inequality:
\begin{equation*} 
\qquad W'(u)(v-u)\ge 0 \quad \text{for all} \quad  v\in  \mathcal{M}(\psi).
\end{equation*}
We find from \cite[Lemma 2]{DG07} that
\begin{equation}\label{eq:2.6}
W'(u)(\varphi)= 
 2\int^1_0 \frac{\kappa (x)}{1+u'(x)^2}\varphi''(x)\, dx
 -5\int^1_0 \frac{\kappa (x)^2 u'(x)}{\sqrt{1+u'(x)^2} }\varphi'(x)\, dx 
\end{equation}
for all $\varphi \in H^2_0(0,1)$}.
\end{remark}

\bigskip 
\noindent
In order to construct admissible obstacles we first recall the scaling behaviour of the one-dimensional elastica functional.
For $v : [0,1] \to \mathbb{R}$ and $\rho>0$, we consider $v_\rho : [0, \rho] \to \mathbb{R}$ defined by $v_\rho(x):= \rho v (x/\rho)$. 
Then it holds that 
\begin{align*}
v'_\rho(x) = v'\left(\dfrac{x}{\rho}\right), \quad v''_\rho(x) = \dfrac{1}{\rho} v''\left(\dfrac{x}{\rho}\right), 
\end{align*}
and then 
\begin{align*}
\kappa_\rho(x)= \dfrac{v''_\rho(x)}{(1+v'_\rho(x)^2)^{3/2}} 
               = \dfrac{1}{\rho} \dfrac{v''(\frac{x}{\rho})}{(1+v'(\frac{x}{\rho})^2)^{3/2}} 
               = \dfrac{1}{\rho} \kappa\left(\dfrac{x}{\rho}\right). 
\end{align*}
Hence we have 
\begin{align*}
W (v_\rho|_{[0, \rho]}) &= \int^\rho_0 \kappa_\rho(x)^2 \sqrt{1 + v'_\rho(x)^2} \, dx \\
                      &= \dfrac{1}{\rho^2} \int^\rho_0 \kappa\left(\dfrac{x}{\rho}\right)^2 \sqrt{1 + v'\left(\dfrac{x}{\rho}\right)^2} \, dx
                       = \dfrac{1}{\rho} W(v). 
\end{align*}
This scaling behaviour of the one-dimensional elastica functional applies to any curve and is not restricted to graphs.

The idea is now in order to find admissible obstacles  to glue different pieces of explicit solutions together and to rescale them.
\begin{remark}\label{theorem:2.4}{\rm \ 
For each $c \in (0, c_0)$, we  let $u_c : [0,1] \to \mathbb{R}$ denote the solution of \eqref{eq:2.1} obtained by Lemma \ref{theorem:2.1}. 
Thanks to Lemma \ref{theorem:2.1}, we see that $W(u_c) = c^2$. 
Here we oddly extend $u_c$ to  $U_c \in H^2( [-1/2, 3/2] , \mathbb{R})$ as follows: 
\begin{align*}
U_c(x) := 
\begin{cases}
-u_c(-x) & \quad x \in [-\frac{1}{2}, 0], \\
u_c(x), & \quad x \in [0,1], \\
-u_c(2-x) & \quad x \in [1, \frac{3}{2}]. 
\end{cases}
\end{align*}
One should observe that only for $c=c_0$ this extension yields a solution to \eqref{eq:2.1}
also for $x$ around $0$ and $1$.

Since $U_c$ is an odd extension of $u_c$, it is clear that $W(U_c) = 2 c^2$. 
By means of a scaling and translation, we define $\hat{u}_c : [0, 1] \to \mathbb{R}$ as follows: 
\begin{align*}
\hat{u}_c(x):= \dfrac{1}{2}U_c(-\tfrac{1}{2} + 2x) - \dfrac{1}{2} U_c(-\tfrac{1}{2}). 
\end{align*}
Recalling  the scaling property of $W$, we observe that 
\begin{align*}
W(\hat{u}_c) = 2 W(U_c) = 4 c^2. 
\end{align*}
Thus, for any $c \in (0, c_0)$, we find that $W(\hat{u}_c)< 4 c_0^2$. 
This implies that for any sufficiently small $\varepsilon >0$, any $\psi \in C^0([0,1])$ with $\psi \le \hat{u}_c-\varepsilon $ obeys conditions~\eqref{eq:A} and~\eqref{eq:B}. 
In particular, any $\hat{u}_c-\varepsilon$  with $c\in (0, c_0)$ is itself an admissible obstacle.

\begin{figure}[h] 
\centering 
\includegraphics[width=.45\textwidth]{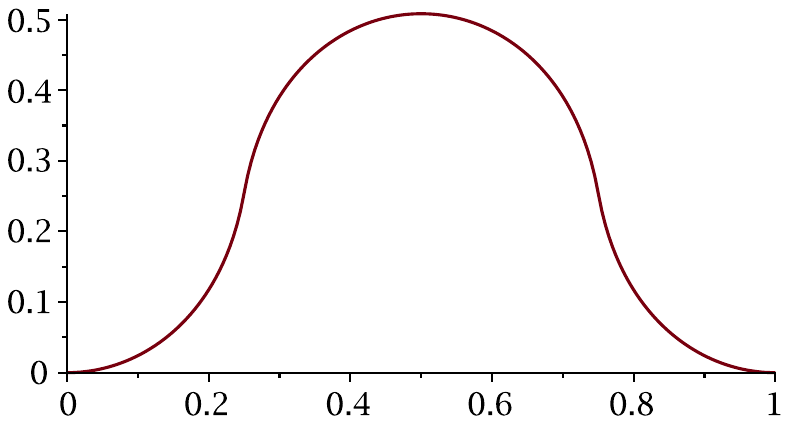}\hspace{0.3cm}
\includegraphics[width=.45\textwidth]{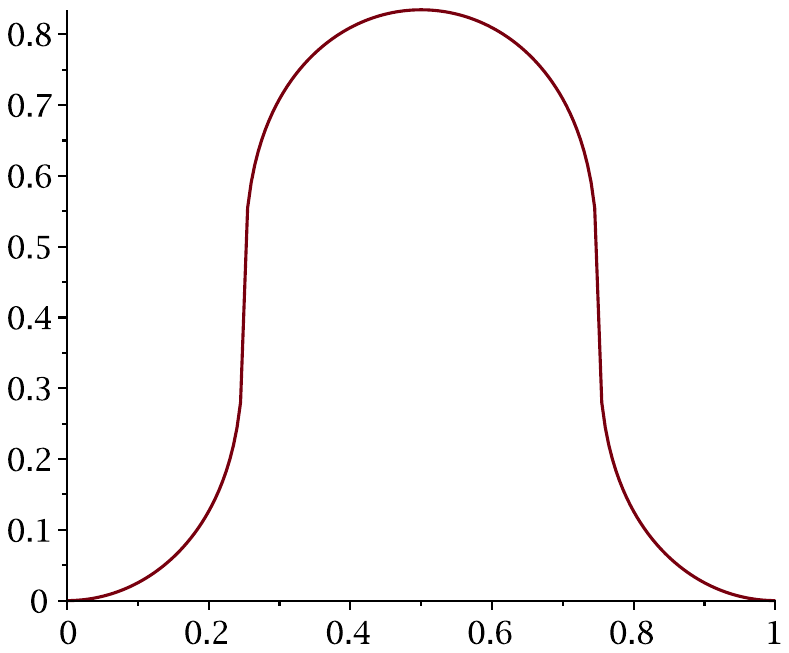}\vspace{-4cm}
\caption{$\hat{u}_{2.3}$  (left) and $\hat{u}_{c_0}$  (right)}
\label{figure:1} 
\end{figure} 

One may observe that
$$
\max_{x\in[0,1]} \hat{u}_{c_0}=\hat{u}_{c_0} (1/2)
= \dfrac{2}{c_0} 
=0.8346268418\ldots.
$$

We consider only $c>0$ in order to obtain obstacles with $\psi(x_0)>0$, i.e. satisfying condition \eqref{eq:A}.
}
\end{remark}

\subsection{Regularity of symmetric minimisers} \label{subsection:2.3}

In order to show regularity we follow the reasoning of Dall'Acqua and Deckelnick, one may see \cite[Proof of Theorem~5.1]{DD}.

Let $u \in \mathcal{M}(\psi)$ be a minimiser of $W(\cdot)$. 
Recalling Remark~\ref{theorem:2.3} we notice first  that $u$ is a weak supersolution of the elastica equation. 
\begin{corollary}\label{theorem:2.9}
	For all $\varphi\in H^2_0(0,1)$ with $\varphi\ge 0$ we have that
	\begin{equation*}
	W'(u)(\varphi)\ge 0.
	\end{equation*}	
\end{corollary}	
\begin{proof}
	Take $u+\varphi$ as comparison function in  Remark~\ref{theorem:2.3}.	
\end{proof}
This corollary shows  that $W'(u)$ is a nonnegative distribution on $C^\infty_{\rm c}(0,1)$. 
This nonnegativity yields that $W'(u)$ is even a  distribution on $C^0_{\rm c}(0,1)$.
Thus, by the Riesz representation theorem we find a nonnegative Radon measure $\mu$ such that 
\begin{equation} \label{eq:2.7}
W'(u)(\varphi)= \int^1_0 \varphi \, d \mu
\end{equation}
for all $\varphi \in C^\infty_{\rm c}(0,1)$. See \cite[Lemma 37.2]{Tartar}.
We define $\mathcal{N} \subset (0,1)$ by 
\begin{equation*}
\mathcal{N} := \{ x \in (0,1) \mid u(x) > \psi(x) \}. 
\end{equation*}
Since $u$ and $\psi$ are continuous in $[0,1]$, we see that the set $\mathcal{N}$ is an open, and 
\begin{equation}\label{eq:important_property}
\mu(\mathcal{N})=0. 
\end{equation}
In other words, the restriction $u|_\mathcal{N}$ is even a \emph{solution} of the 
elastica equation.

Indeed, for all $\varphi \in C^\infty_{\rm c}(\mathcal{N})$, we have $u + \varepsilon \varphi \ge \psi$ in $[0,1]$ for $\varepsilon>0$ small enough, 
and then $W'(u)(\varphi)=0$. This implies that $u|_{\mathcal{N}}$ is even a weak \emph{solution}
of the elastica equation.

\begin{lemma} \label{theorem:2.5}
Assume that $\psi$ satisfies condition~\eqref{eq:A}. Suppose that there exists a minimiser $u \in \mathcal{M}(\psi)$  of $W(\cdot)$. 
Then there exist $a \in (0,1/2)$ such that $(0,a) \cup (1-a,1) \subset \mathcal{N}$. 
\end{lemma}
\begin{proof}
Since $\psi$ is continuous in $[0,1]$, there exists $\delta>0$ such that 
\begin{equation*}
\psi(x) < \dfrac{3}{4} \psi(0) \quad \text{for all} \quad x \in [0, \delta). 
\end{equation*}
By Sobolev's embedding $H^2(0,1) \subset W^{1, \infty}(0,1)$ and since $\mathcal{M}(\psi)
\subset H^2(0,1)$, we find a constant $C_1>0$ such that  
\begin{equation*}
\| u \|_{W^{1, \infty}(0,1)} \le C \| u \|_{H^2(0,1)} \le C_1. 
\end{equation*}
This together with $u(0) =0$ implies that  
\begin{equation*}
|u(x)| \le C_1 x \quad \text{for all} \quad x \in [0,1]. 
\end{equation*}
Thus, taking $\delta'>0$ such that $\delta'< -\psi(0)/4 C_1$, we see that 
\begin{align*}
u(x) \ge \dfrac{1}{4} \psi(0) \quad \text{for all} \quad x \in [0, \delta']. 
\end{align*}
Setting $a:= \min\{ \delta, \delta'\}$, we deduce that $(0,a) \subset \mathcal{N}$. 
The symmetry of $u$ yields  that also $(1-a,1) \subset \mathcal{N}$. 
The proof of Lemma \ref{theorem:2.5} is complete.  
\end{proof}

We next show that the nonnegative Radon measure $\mu$ is finite.
\begin{lemma} \label{theorem:2.6}
Assume that $\psi$ satisfies condition~\eqref{eq:A}.  
Suppose that there exists a minimiser $u \in \mathcal{M}(\psi)$  of $W(\cdot)$. 
Then
\begin{equation*}
\mu(0,1) <\infty.
\end{equation*}
\end{lemma}
\begin{proof}
Thanks to Lemma \ref{theorem:2.5} and \eqref{eq:important_property}, we find $a\in (0,1/2)$ such that $\mu((0,1)) = \mu([a,1-a])$. 
Fix $\eta \in C^\infty_{\rm c}(0,1)$ with $\eta \equiv 1$ in $[a,1-a]$ and $0 \le \eta \le 1$ in $(0,1)$. Since by assumption $u \in \mathcal{M}(\psi)\subset H^2(0,1)\hookrightarrow
C^1 ([0,1])$  we conclude from \eqref{eq:2.6}:
\begin{align*}
\mu((0,1)) &=\mu([a,1-a])\le \int^1_0 \eta(x) \, d\mu 
                  = W'(u)(\eta) \\
               & \le C (\| u''\|_{L^2(0,1)} \| \eta'' \|_{L^2(0,1)} + \| u''\|^2_{L^2(0,1)} \| \eta' \|_{L^\infty(0,1)}) \\
               & \le C ( \| \eta'' \|_{L^2(0,1)} + \| \eta' \|_{L^\infty(0,1)}). 
\end{align*}
This proves the claim. 
\end{proof}

In order to study the regularity of minimisers, we employ  ideas used in \cite[Proposition 3.2]{DD} and \cite[Theorem 3.9]{DDG}. 
\begin{lemma} \label{theorem:2.7}
Fix $\eta \in C^\infty_{\rm c}(0,1)$ and set 
\begin{align*}
\varphi_1(x) &:= \int^x_0 \!\! \int^y_0 \eta(s) \, ds dy + \alpha x^2 + \beta x^3, \\
\varphi_2(x) &:= \int^x_0 \eta(y) \, dy + (-3x^2 + 2 x^3) \int^1_0 \eta(y) \, dy, 
\end{align*}
for $x \in [0,1]$, where 
\begin{align*}
\alpha := \int^1_0 \eta(y) \, dy - 3 \int^1_0 \!\! \int^y_0 \eta(s) \, ds dy, \qquad 
\beta := -\alpha - \int^1_0 \!\! \int^y_0 \eta(s) \, ds dy.
\end{align*}
Then, $\varphi_1, \varphi_2 \in H^2_0(0,1)$ and there exists $C>0$ such that 
\begin{gather*}
\| \varphi_1 \|_{C^1(0,1)}, |\alpha|, |\beta| \le C \| \eta \|_{L^1(0,1)}, \\
\| \varphi_2 \|_{L^\infty(0,1)} \le C \| \eta\|_{L^1(0,1)}, \qquad \| \varphi'_2 \|_{L^p(0,1)} \le C \| \eta \|_{L^p(0,1)} \quad \text{for} \quad p \in [1, \infty). 
\end{gather*}
\end{lemma}

\begin{proposition} \label{theorem:2.8}
Assume that $\psi$ satisfies condition~\eqref{eq:A}. 
Suppose that there exists a minimiser $u \in \mathcal{M}(\psi)$  of $W(\cdot)$. 
Then $u\in C^2([0,1])$, $u''$ is weakly differentiable and $u'''\in BV(0,1)$. 
On the complement of the coincidence set the minimiser is smooth, i.e. $u|_{\mathcal N}\in C^\infty({\mathcal N})$.
\end{proposition}
\begin{proof}
We define a function $m : (0,1) \to \mathbb{R}$ by 
\begin{equation*}
m(x) = \mu(0,x) \quad \text{for} \quad x \in (0,1). 
\end{equation*}
Then $m$ is increasing and  bounded  on $(0,1)$, and  the Lebesgue--Stieltjes integral induced by $m$ is well-defined. 
Using integration by parts for Lebesgue--Stieltjes integrals (\cite[Chapter III, Theorem 14.1]{Saks}), we obtain 
\begin{equation} \label{eq:2.8}
\int^1_0 \varphi \, d \mu(x) = - \int ^1_0 m(x) \varphi'(x) \, dx 
\end{equation}
for all $\varphi \in C^\infty_{\rm c}(0,1)$. 
It follows from \eqref{eq:2.6}, \eqref{eq:2.7} and \eqref{eq:2.8} that 
\begin{equation}
\label{eq:2.9}
2\int^1_0 \frac{\kappa (x)}{1+u'(x)^2}\varphi''(x)\, dx
= 5\int^1_0 \frac{\kappa (x)^2 u'(x)}{\sqrt{1+u'(x)^2} }\varphi'(x)\, dx  - \int ^1_0 m(x) \varphi'(x) \, dx  
\end{equation}
for all $\varphi \in C^\infty_{\rm c}(0,1)$. 
By a density argument, we see that \eqref{eq:2.9} also holds for all $\varphi \in H^2_0(0,1)$. 

Fix $\eta \in C^\infty_{\rm c}(0,1)$ arbitrarily. 
Taking $\varphi_1$ as $\varphi$ in \eqref{eq:2.9}, where $\varphi_1$ is defined in Lemma~\ref{theorem:2.7}, we have 
\begin{equation}
\label{eq:2.10}
\begin{aligned}
2 \int^1_0 \dfrac{u''(x)}{(1+u'(x)^2)^{5/2}} \eta(x) \, dx 
&= -4 \int^1_0 \dfrac{u''(x)}{(1+u'(x)^2)^{5/2}} (\alpha + 3 \beta x) \, dx \\
& \qquad + 5 \int^1_0 \dfrac{|u''(x)|^2 u'(x)}{(1+u'(x)^2)^{7/2}} \varphi'_1(x) \, dx \\
& \qquad \quad  - \int ^1_0 m(x) \varphi'_1(x) \, dx 
=: I_1 + I_2 + I_3. 
\end{aligned}
\end{equation}
Since by assumption $u \in \mathcal{M}(\psi)\subset H^2(0,1)\hookrightarrow
C^1 ([0,1])$, we observe from Lemmas~\ref{theorem:2.6} and~\ref{theorem:2.7} that  
\begin{align*}
|I_1| & \le (4 |\alpha| + 12|\beta|) \| u'' \|_{L^1(0,1)} \le C \| \eta\|_{L^1(0,1)}, \\
|I_2| & \le 5 \| u'' \|^2_{L^2(0,1)} \|\varphi_1 \|_{C^1([0,1])} \le C \| \eta \|_{L^1(0,1)}, \\
|I_3| & \le \sup_{x \in [0,1]} m(x) \| \varphi_1 \|_{C^1([0,1])} \le C \mu(0,1) \| \eta \|_{L^1(0,1)} \le C \| \eta \|_{L^1(0,1)}. 
\end{align*}
This together with \eqref{eq:2.10} implies that 
\begin{equation}
\label{eq:2.11}
\| u''(x) (1+u'(x)^2)^{-5/2} \|_{L^\infty(0,1)} \le C. 
\end{equation}
In view of $u \in C^1 ([0,1])$  we obtain from \eqref{eq:2.11}:
\begin{equation}
\label{eq:2.12}
\| u'' \|_{L^\infty(0,1)} \le C. 
\end{equation}

Fix $\eta \in C^\infty_{\rm c}(0,1)$ arbitrarily. 
Taking $\varphi_2$ as $\varphi$ in \eqref{eq:2.9}, where $\varphi_2$ is defined by Lemma \ref{theorem:2.7}, we have 
\begin{align*}
\int^1_0 \dfrac{u''(x)}{(1+u'(x)^2)^{5/2}} \eta'(x) \, dx 
 &= 6 \left(\int^1_0 (1-2x) \dfrac{u''(x)}{(1+u'(x)^2)^{5/2}} \, dx\right)\left( \int^1_0 \eta(x) \, dx\right) \\
 & \qquad + \dfrac{5}{2} \int^1_0 \dfrac{u''(x)^2 u'(x) \varphi'_2(x)}{(1+u'(x)^2)^{7/2}} \, dx \\
 & \qquad \quad - \dfrac{1}{2} \int^1_0 m(x) \varphi'_2(x)\, dx
  =: I'_1 + I'_2  + I'_3. 
\end{align*}
We deduce from  \eqref{eq:2.12} and  $u \in  H^2(0,1)\hookrightarrow
C^1 ([0,1])$ that 
\begin{align*}
|I'_2| \le C \| u''\|_{L^\infty(0,1)}^2 \| \varphi'_2\|_{L^1(0,1)} 
       \le C \| \eta \|_{L^1(0,1)}. 
\end{align*}
Along the same lines as above, we obtain 
\begin{align*}
|I'_1| &\le 6 \| u''\|_{L^\infty(0,1)} \| \eta\|_{L^1(0,1)} \le C \| \eta\|_{L^1(0,1)}, \\
|I'_3| &\le \dfrac{1}{2}\Bigl( \sup_{x \in [0,1]} m(x) \Bigr)\, \| \varphi'_2\|_{L^1(0,1)} \le C \mu(0,1) \|\eta\|_{L^1(0,1)} \le C \|\eta\|_{L^1(0,1)}.  
\end{align*}
Thus we observe that 
\begin{equation*}
\Bigl| \int^1_0 \dfrac{u''(x)}{(1+u'(x)^2)^{5/2}} \eta'(x) \, dx \Bigr| \le C \| \eta\|_{L^1(0,1)} 
\end{equation*}
for all $\eta \in C^\infty_{\rm c}(0,1)$, and then 
\begin{equation*}
\| (u'' (1 + (u')^2)^{-5/2})' \|_{L^{\infty}(0,1)} \le C. 
\end{equation*}
This together with  $u \in C^1 ([0,1])$ and \eqref{eq:2.12} implies that 
\begin{equation}
\label{eq:2.13}
\| u'''\|_{L^{\infty}(0,1)} <\infty. 
\end{equation}
We observe from \eqref{eq:2.9} and \eqref{eq:2.13} that 
for all  $\varphi \in C^\infty_{\rm c}(0,1)$ we have
\begin{align*}
\int^1_0 \Bigl[ -2 \Bigl(\frac{u''(x)}{(1+u'(x)^2)^{5/2}}\Bigr)' - 5
 \frac{u''(x)^2 u'(x)}{(1+u'(x)^2)^{7/2}} +m(x) \Bigr] \varphi'(x)\, dx 
 =0.  
\end{align*}
Thus there exists a constant $c \in \mathbb{R}$ such that 
\begin{align*}
-2 \Bigl(\frac{u''(x)}{(1+u'(x)^2)^{5/2}}\Bigr)' - 5 \frac{u''(x)^2
	 u'(x)}{(1+u'(x)^2)^{7/2}} + m(x) = c \quad \text{for all} \quad x \in (0,1). 
\end{align*}
Recalling that $m(\cdot)$ is of bounded variation, we see that $u''' \in BV(0,1)$. 

Finally, $u|_{\mathcal N}$ is a solution of the elastica equation \eqref{eq:2.1} on the complement $\mathcal N$
of the coincidence set. Further boot strapping yields $u|_{\mathcal N}\in C^\infty({\mathcal N})$.
The proof of Proposition~\ref{theorem:2.8} is complete. 
\end{proof}

\bigskip

\subsection{Further properties of minimisers in the symmetric case} \label{subsection:2.4}

We deduce some properties of solutions of the obstacle problem as constructed in the previous section. 
That means that in what follows we consider 
\begin{equation*}
u \in \mathcal{M}(\psi):\quad 
W(u) = \inf_{v \in \mathcal{M}(\psi)} W(v). 
\end{equation*}

We recall that according to condition~\eqref{eq:A} we always have that 
\begin{equation*}
\exists x_0\in (0,1):\quad \psi (x_0) >0. 
\end{equation*}
So, for sure, the solution $u$ is strictly positive somewhere. 
In what follows we shall find out much more. 
Among others we prove Theorem~\ref{theorem:1.2} which implies  that for obstacles exceeding a specific universal bound our minimisation problem has \emph{no solution}.

One should always have in mind that the following results are valid due to the fact that we are looking for minimisers in the class of ``relatively smooth'' symmetric
\emph{graphs}.
In this class existence results are more restrictive while we have more and stronger results concerning qualitative properties of minimisers. 
We expect that this situation will change fundamentally when admitting all ``sufficiently smooth'' \emph{curves}. 
We emphasise that smooth curves, which are graphs, need not be smooth graphs.
As for this one should always have the function $\hat{u}_{c_0}$ in mind, which is defined in Remark~\ref{theorem:2.4}.


We recall from Corollary~\ref{theorem:2.9} that the minimiser $u$ under consideration is a weak supersolution of the elastica equation, i.e. 
	for all $\varphi\in H^2_0(0,1)$ with $\varphi\ge 0$ we have that
$
W'(u)(\varphi)\ge 0.
$
Beside the regularity result in Proposition~\ref{theorem:2.8} only this is exploited in what follows.

We  associate with  $u$ the following corresponding auxiliary function:
\begin{equation*}
V(x):= V_u(x):=\kappa_u(x) (1+u'(x)^2)^{1/4}=\frac{u''(x)}{(1+u'(x)^2)^{5/4}}.
\end{equation*}
This turned to out to be extremely useful in \cite{DG07}, see e.g. Lemma 3 therein. 
Obviously, already Euler \cite{Euler} was aware of the importance of this function. 
Similarly as in \cite{DG07} we see that this function is a weak subsolution of a second order differential equation in divergence form:

\begin{proposition}\label{theorem:2.10}
We have for all $\varphi\in H^2_0(0,1)$ with $\varphi\ge 0$ that 
\begin{equation}\label{eq:2.14}
0\ge  \int^1_0 \varphi'(x) \, \left(\frac{V'(x)}{(1+u'(x)^2)^{5/4}} \right)\, dx.
\end{equation}
Moreover, one has
\begin{equation}\label{eq:2.15}
0\ge  \int^1_0 \left( \varphi(x) u'(x)\right)' \, \left(\frac{u'(x) V'(x)}{(1+u'(x)^2)^{5/4}} -\frac12 V(x)^2\right)\, dx.
\end{equation}
\end{proposition}
That the right hand side of (\ref{eq:2.14}) is the weak form of the elastica operator was already observed by M\"uller \cite[Proposition 2.3]{Mueller_2019}.
\begin{proof}
We conclude from  Remark~\ref{theorem:2.3}, Corollary~\ref{theorem:2.9}, and Proposition~\ref{theorem:2.8}  that the following holds for all such $\varphi$:	
$$
0\ge \int^1_0 \varphi'(x) \left( 
2 \left( \frac{\kappa (x)}{1+u'(x)^2}\right)'
+5 \frac{\kappa (x)^2 u'(x)}{\sqrt{1+u'(x)^2} }\right) \, dx.
$$
Observing that
\begin{align*}
2 \left( \frac{\kappa (x)}{1+u'(x)^2}\right)'
&=2 \left( \frac{V (x)}{(1+u'(x)^2)^{5/4}}\right)'
=2  \frac{V' (x)}{(1+u'(x)^2)^{5/4}}
  -5 \frac{V (x)u'(x)u''(x)}{(1+u'(x)^2)^{9/4}}\\
 & =2  \frac{V' (x)}{(1+u'(x)^2)^{5/4}}
  -5 \frac{V (x)^2u'(x)}{1+u'(x)^2}
\end{align*}
we find
$$
0\ge 2 \int^1_0 \varphi'(x)\frac{V' (x)}{(1+u'(x)^2)^{5/4}}\, dx.
$$
This shows the first claim. In order to see the second one we use $\varphi (u')^2$ as a test function and observe that 
$\left( \varphi (u')^2\right)' = u'\left( \varphi u'\right)'
+\varphi \, u'\, u'' $:
\begin{align*}
0& \ge  \int^1_0 \left( \varphi (u')^2\right)'(x) \, \left(\frac{V'(x)}{(1+u'(x)^2)^{5/4}} \right)\, dx\\
&=  \int^1_0 \left( \varphi u'\right)'(x) \, \left(\frac{u'(x) V'(x)}{(1+u'(x)^2)^{5/4}} \right)\, dx
+\int^1_0 \varphi (x) u'(x) V(x) V'(x)\, dx\\
&= \int^1_0 \left( \varphi u'\right)'(x) \, \left(\frac{u'(x) V'(x)}{(1+u'(x)^2)^{5/4}} \right)\, dx
+\int^1_0 \varphi (x) u'(x)\left(\frac{V^2}{2} \right)'(x)\, dx,
\end{align*}
and also the second claim \eqref{eq:2.15} follows.
\end{proof}

\medskip 
The previous proposition reflects on the one hand the well known fact that the elastica operator as a whole is of divergence form.
On the other hand this shows that $V$ is a subsolution of a second order differential inequality without zeroth-order term and hence obeys a strong Hopf-type maximum principle.

\begin{proposition}\label{theorem:2.11}
The auxiliary function $V$ satisfies $V'(x)<0$ on $[0,1/2)$ and $V'(x)>0$ on $(1/2,1]$ and obeys
$$
V(0)=V(1)>0,\qquad V(\frac12) <0.
$$
There exists a point $a\in (0,1/2)$ such that $V>0$ (and hence $u$ is strictly convex) on $[0,a)\cup (1-a,1]$ and $V<0$ (and hence $u$ is strictly concave) on $(a,1-a)$. 
Moreover we have 
\begin{equation}\label{eq:2.16.0}
\forall x\in (0,\frac12):\quad u'(x)>0,\qquad 
\forall x\in (\frac12,1):\quad u'(x)<0;
\end{equation}
\begin{equation}\label{eq:2.16}
\forall x\in [0,\frac12):\quad \kappa'(x)<0,\qquad 
\forall x\in (\frac12,1]:\quad \kappa'(x)>0.
\end{equation}
Finally we have
$$
\forall x\in (\frac12,1):\quad
V(1/2)^2 \le 
 \left(V(x)^2-2\frac{u'(x) V'(x)}{(1+u'(x)^2)^{5/4}} \right)
 \le V(1)^2.
$$
In particular
$$
0<-\kappa (1/2) \le \kappa (1).
$$
\end{proposition}

\begin{proof}
Since $V$ obeys  a strong Hopf-type maximum principle and is symmetric around $1/2$ we have
$$
\forall x\in [0,1/2):\quad V'(x)<0,\qquad \forall x\in (1/2,1]:\quad V'(x)>0.
$$
We observe further that by assumption $u$ must be strictly positive somewhere, so
we find a global and local maximum in some $x_0\in(0,1)$. In view of the boundary conditions we find then points
$x_1'\in (0,x_0)$, $x_1\in (x_0,1)$ with $0<u(x_1')<u(x_0)$ and $0<u(x_1)<u(x_0)$. Hence there exist
$x_2'\in (x_1',x_0)$, $x_2\in (x_0,x_1)$ with $0<u(x_2')$, $0<u(x_2)$, $0>u''(x_2')$, and $0>u''(x_2)$.
Using once more the boundary conditions $u(0)=u(1)=u'(0)=u'(1)=0$ we find 
$x_3'\in (0,x_2')$, $x_3\in (x_2,1)$ with $0<u''(x_3')$, and $0<u''(x_3 )$.

This proves that $V$ is strictly positive somewhere as well as strictly negative somewhere else. Hence we find 
$a\in (0,1/2)$ such that $V>0$  on $[0,a)\cup (1-a,1]$ and $V<0$  on $(a,1-a)$. This implies that $u$ is strictly convex
on $[0,a)\cup (1-a,1]$  and   strictly concave on $(a,1-a)$. Since $u''>0 $  on $[0,a)\cup (1-a,1]$, $u''<0$  on $(a,1-a)$, and $u'(0)=u'(1/2)=u'(1)=0$, \eqref{eq:2.16.0} follows.

To prove \eqref{eq:2.16} it suffices to consider $x\in (1/2,1]$
where $u'(x)<0$, because $\kappa$ is symmetric around $1/2$. Here we find from our first observation for $x\in (1/2,1]$:
\begin{align*}
0&<V'(x) = \left( \kappa(x) (1+u'(x)^2)^{1/4}\right)'\\
  & =\kappa'(x) (1+u'(x)^2)^{1/4} +\frac12 \kappa(x) (1+u'(x)^2)^{-3/4}u'(x) u''(x)\\
   &=\kappa'(x) (1+u'(x)^2)^{1/4} +\frac12 \kappa(x)^2 (1+u'(x)^2)^{3/4}u'(x)\\
   &\le \kappa'(x) (1+u'(x)^2)^{1/4},
\end{align*}
and we obtain \eqref{eq:2.16}. In view of 
(\ref{eq:2.15}) the last results follow.
\end{proof}

\bigskip

Although higher order equations do in general not obey any kind of comparison principle, 
here we can show that symmetric supersolutions lie above symmetric solutions which obey the same Dirichlet boundary conditions.

\begin{lemma}\label{theorem:2.12}
Assume that $u\in C^2([0,1])\cap W^{3,\infty}(0,1)$ is symmetric around $1/2$, satisfies for some $\beta \in \mathbb{R}$ Dirichlet boundary conditions
$$
u(0)=u(1)=0,\qquad u'(0)=-u'(1)=\beta
$$
and is a supersolution of the elastica equation in the  sense that we have
for all $\varphi\in H^2_0(0,1)$ with $\varphi\ge 0$ that 
\begin{equation*}
0\ge  \int^1_0 \varphi'(x) \, \left(\frac{V'(x)}{(1+u'(x)^2)^{5/4}} \right)\, dx.
\end{equation*}
Let $w\in C^4([0,1])$ be a solution of the elastica equation
$$
\dfrac{1}{\sqrt{1 + w'(x)^2}} \dfrac{d}{dx} \left( \dfrac{\kappa_w'(x)}{\sqrt{1 + w'(x)^2}} \right) + \dfrac{1}{2} \kappa_w(x)^3 = 0, \quad x \in (0,1).
$$
which is symmetric around $1/2$ and satisfies the same Dirichlet boundary conditions
$$
w(0)=w(1)=0,\qquad w'(0)=-w'(1)=\beta.
$$
Then the following comparison statement holds:
$$
\forall x\in (0,1):\qquad u(x)\ge w(x).
$$
Moreover we either have equality or  strict inequality everywhere.
\end{lemma}
\begin{proof} 
We define 
$$c:=-2G(\beta).$$ 
According to \cite[Theorem 2]{DG07} the solution $w$ is unique in the class of symmetric smooth graphs and satisfies
\begin{equation*}
w'(x)=G^{-1} \left(c(x-1/2)\right),
\end{equation*}
cf. also Lemma~\ref{theorem:2.1}.
We let $V=V_u=\frac{u''}{(1+(u')^2)^{5/4}}$ denote the usual auxiliary function and show first that
\begin{equation}\label{eq:2.17}
V(0)\ge c
\end{equation}
and assume by contradiction that 
$$
V(0)<c.
$$
We only used that $u$ is a sufficiently smooth supersolution to show Proposition~\ref{theorem:2.10}. 
Hence the strong maximum principle applies to $V$ and yields that
$$
\forall x\in [0,1]: V(x) < c.
$$
Observe that the symmetry of $u$ yields that $u'(1/2)=0$. 
Integration yields
\begin{align*}
c&>2\int^{1/2}_0 \frac{u''(\xi)}{(1+u'(\xi)^2)^{5/4}}\, d\xi
=2\int^{u'(1/2)=0}_{u'(0)=\beta}\frac{1}{(1+\tau^2)^{5/4}}\, d\tau \\
&=2G(0)-2G(\beta)=c,
\end{align*}
a contradiction. 
This proves (\ref{eq:2.17}).

In case that $V(x)\equiv c$ integration as in \cite[Theorem 2]{DG07} shows that $u(x)\equiv w(x)$. 
Otherwise we have thanks to the strict maximum principle that either 
$$
\forall x\in (0,1)\setminus\{1/2\}:\qquad V(x)>c.
$$
or
$$
\exists x_0\in [0,1/2): \quad 
\forall x\in (0,x_0)\cup (1-x_0):\quad V(x)>c
$$
and
$$
\forall x\in (x_0,1-x_0):\quad V(x)<c.
$$
In view of the first part of the proof it follows that $x_0>0$.
In the first situation we conclude for all $x\in (0,1/2]$ that
\begin{align*}
cx&<\int^{x}_0 \frac{u''(\xi)}{(1+u'(\xi)^2)^{5/4}}\, d\xi
=\int^{u'(x)}_{\beta}\frac{1}{(1+\tau^2)^{5/4}}\, d\tau 
=G(u'(x))+\frac{c}{2},\\
\Rightarrow\ & w'(x) < u'(x).
\end{align*}
In particular, $0=w'(1/2) < u'(1/2)=0$, which is impossible.
So we are left with the second situation where we obtain in the same way as before that
$$
\forall x\in (0,x_0):\quad w'(x) < u'(x).
$$
For $x\in (x_0,1/2)$ we integrate instead on $(x,1/2)$ and find thanks to the symmetry of $u$:
\begin{align*}
c\left(1/2-x\right)&>\int_{x}^{1/2} \frac{u''(\xi)}{(1+u'(\xi)^2)^{5/4}}\, d\xi
=\int_{u'(x)}^{u'(1/2)=0}\frac{1}{(1+\tau^2)^{5/4}}\, d\tau 
=-G(u'(x)),\\
\Rightarrow\ & w'(x) < u'(x).
\end{align*}
So we conclude that
$$
\forall x\in (0,1/2)\setminus \{x_0\}:\quad  w'(x) < u'(x).
$$
Hence for any $x\in (0,1)$ it holds that $u(x)>w(x)$ except when $V(x)\equiv c$ and hence $u(x)\equiv w(x)$.
\end{proof}

\begin{remark}\label{theorem:2.13}
	{\rm \
    Since $\beta$ was arbitrary in the previous lemma and multiplication
    by $-1$ changes a subsolution into a supersolution, an analogous statement
    holds also for subsolutions.
    }
\end{remark}

\bigskip 
\noindent
For what follows and up to the end of this subsection we fix the following notation: 
$u\in C^2([0,1])\cap W^{3,\infty}(0,1)$ with $u'''\in BV(0,1)$ always denotes a supersolution (which is not a solution) of the elastica equation which is symmetric around $1/2$ 
and satisfies homogeneous Dirichlet boundary conditions
$$
u(0)=u(1)=0,\qquad u'(0)=-u'(1)=0.
$$
We call the  reparameterisation by arclength of the graph of $u$
$$
y:[-L,L]\to [0,1]\times \mathbb{R}, \quad L=\frac12 \int^1_0 \sqrt{1+u'(\xi)^2}\, d\xi
$$
so that 
$$
\forall x\in [0,1]:\quad y\left(-L+ \int^x_0 \sqrt{1+u'(\xi)^2}\, d\xi\right)=(x,u(x)).
$$
We let $\kappa:[-L,L]\to  \mathbb{R}$ denote its curvature function in arclength parameterisation,
$$
\tau, n:[-L,L]\to  \mathbb{R}^2,\quad \tau (s):= y_s(s),\quad 
n(s):=\begin{pmatrix} -\tau^2\\ \tau^1 \end{pmatrix}
$$
its unit tangent and normal, respectively.
Being a supersolution of the elastica equation means that
\begin{equation*}
\kappa_{ss} +\frac12\kappa^3 \ge 0.
\end{equation*}
According to Proposition~\ref{theorem:2.11} we know that
\begin{equation}\label{eq:2.18}
\forall s\in (0,L] :\quad \kappa_s(s) >0.
\end{equation}

We shall compare $u$ with a suitable solution $w$ of the elastica equation, which is symmetric around $1/2$ and which will be specified later.
One should have in mind that such a solution is smooth as a curve, it is a graph, but it may have infinite slope and is hence there not a smooth graph. 
At points where $w$ has infinite slope the corresponding auxiliary function $V_w$ may  have even jump discontinuouties.

In any case we shall assume in $x=1/2$,  that $w$ has its maximum and its curvature its minimum there.
Let $Y$ denote the arclength parameterisation of its graph (which is smooth) with $Y(0)=(1/2, w(1/2))$ and $k,T,N$ denote its curvature function, unit tangent and normal field, respectively.
\begin{lemma}\label{theorem:2.14}
In the situation just described we assume further  that 
$$
0> \kappa(0)=k(0).
$$
Then
\begin{equation*}
\forall s\in [0,L]:\quad \kappa(s)\ge k(s).
\end{equation*}
\end{lemma}
\begin{proof}
We assume first for some  positive $\varepsilon$ close to zero that 
$$
 \kappa_{ss} +\frac12\kappa^3 \ge \varepsilon > 0
$$
and prove that then even
\begin{equation}\label{eq:2.19}
\forall s\in (0,L]:\quad \kappa(s) > k(s).
\end{equation}	
Since $\kappa$ is continuous  we first see for $s$ close enough to $0$ that
$$
\kappa_{ss}(s) \ge  \varepsilon-\frac12\kappa(s)^3 \ge \frac34 \varepsilon-\frac12\kappa(0)^3 
= \frac34 \varepsilon-\frac12 k(0)^3
\ge \frac12 \varepsilon-\frac12 k(s)^3 = \frac12 \varepsilon + k_{ss}(s).
$$ 
This differential inequality has to be understood in a weak sense. Since $\kappa(\,.\,) $ and $k(\,.\,) $
are even functions we find from the weak maximum principle that \eqref{eq:2.19} is satisfied for small positive $s$. 
We assume by contradiction that this is not the case on the whole interval and  choose $s_0\in (0,L]$ minimal with
$$
\kappa(s_0) = k(s_0), \mbox{\ so that\ } \forall s\in (0,s_0):\quad \kappa(s) > k(s).
$$
This implies that 
\begin{equation}\label{eq:2.20}
0< \kappa_s (s_0) \le k_s (s_0).
\end{equation}
We obtain from the differential inequality and equation, respectively, and from 
\eqref{eq:2.18} that for all $s\in (0,L]$:
$$
\frac{d}{ds}\left(4\kappa_s^2+\kappa^4 \right)=8\kappa_s \left(\kappa_{ss} +\frac12\kappa^3  \right) > 0
=\frac{d}{ds}\left(4 k_s^2+ k^4 \right).
$$
Integrating over $[0,s_0]$ yields
$$
\left[4\kappa_s^2+\kappa^4 \right]^{s_0}_0
> \left[4 k_s^2+ k^4 \right]^{s_0}_0.
$$
Making use  of $\kappa_s(0)=k_s(0)=0$, of $\kappa(s_0)=k(s_0)$, of $\kappa(0)=k(0)$ and of \eqref{eq:2.20} we find
$$
4\kappa_s(s_0)^2 > 4 k_s(s_0)^2
\ge 4\kappa_s(s_0)^2,
$$
a contradiction.

In the general case that  we apply  first \eqref{eq:2.19} to $\kappa +\frac\varepsilon2 s^2$ and let then $\varepsilon\searrow 0$. 
\end{proof}

\bigskip 
\noindent 
\begin{proof}[Proof of Theorem~\ref{theorem:1.2}]
We introduce angle functions
$$
\alpha: [-L,L] \to \left( -\frac{\pi}{2},\frac{\pi}{2}\right),\qquad 
A: [-L,L] \to \left[ -\frac{\pi}{2},\frac{\pi}{2}\right]
$$
such that 
\begin{equation*}
\tau(s)=\begin{pmatrix}
\cos (\alpha(s)) \\
\sin (\alpha (s) )
\end{pmatrix},\qquad 
T(s)=\begin{pmatrix}
\cos (A(s)) \\
\sin (A(s) )
\end{pmatrix}.
\end{equation*}
We take from planar differential geometry that
$$
\forall s\in [-L,L]:\quad \kappa (s)=\dot{\alpha} (s),\quad 
k(s)=\dot{A} (s).
$$
Since $\alpha(0)=A(0)=0$, Lemma~\ref{theorem:2.14} yields that 
\begin{equation*}
\forall s\in [0,L]:\quad \alpha(s) \ge A(s), \quad \alpha(s) \in \Big(-\frac{\pi}{2},0\Big]. 
\end{equation*}
The last claim follows from Proposition~\ref{theorem:2.11}.
This means that on $[0,L]$ both angle functions take values only  in $[-\pi/2,0]$. On this interval  $\cos$ as well as $\sin$ are monotonically increasing so that 
we have componentwise ordering of the tangent vectors:
$$
\forall s\in [0,L]:\quad
\tau (s) \ge T(s), \mbox{\ meaning that for\ }j=1,2:
\quad \tau^j (s) \ge T^j(s).
$$
We specify the solution $w$ such that $w(1/2)=u(1/2)$ and that in arclength reparameterisation $k(s=0)=\kappa(s=0)$ so that in particular $Y(0)=y(0)$. Such a comparison function exists in view 
of \cite{DG07} and the scaling behaviour of the curvature.
We come up with 
\begin{equation*}
\forall s\in [0,L]:\quad y(s) \ge Y(s)
\end{equation*}
to be interpreted compenentwise as before. 
This means that the supersolution lies above and right from the corresponding solution with $k(0)=\kappa(0)$ and  $Y(0)=y(0)$.

This means that the average slope of the supersolution is less than the average slope of the solution, but restricted to $[0,L]$:
$$
2u(1/2)=-\frac{u(1)- u(1/2)}{1-1/2} = - \frac{y^2(L)-y^2(0)}{y^1(L)-y^1(0)}\le -\frac{Y^2(L) -Y^2(0)}{Y^1(L)-Y^1(0)}.
$$
In order to find a bound for the maximal average slope of symmetric solutions 
(which is scaling invariant) one has to maximise (cf. Remark~\ref{theorem:2.4} and Lemma~\ref{theorem:2.1})
$$
[0,1/2] \ni x \mapsto \frac{U_{c_0}(x)-U_{c_0}(-1/2)}{x+1/2}=
\frac{
	\dfrac{2}{c_0 \sqrt[4]{1 + G^{-1}\left( \tfrac{c_0}{2} - c_0 x \right)^2}} + \dfrac{2}{c_0 }
}{
x+\frac12
}.
$$
By substituting $z:=G^{-1}\left( \tfrac{c_0}{2} - c_0 x \right)$
this is equivalent to maximising
$$
[0,\infty] \ni z \mapsto \frac{
2+2(1+z^2)^{-1/4}
}{
	c_0-G(z)
}.
$$
This maximum is given by the unique solution $z_0\in (0,\infty)$ of the equation 
\begin{equation}\label{eq:2.16.a.3}
z_0(c_0-G(z_0))=2+2(1+z_0^2)^{-1/4}.
\end{equation}
We think that this equation cannot be solved explicitly.
With the help of maple\texttrademark\ or mathematica\texttrademark\  one finds that
$z_0 = 2.3780929080\ldots$ is the  solution and so, a bound for this function. 
Hence, \emph{any supersolution} of the elastica equation as considered in this section obeys the \emph{universal bound}:
\begin{equation}\label{eq:2.21}
\forall x\in [0,1]:\quad 
0\le u(x) \le u\left(\frac12\right)\le \frac{z_0}{2}=1.1890464540\ldots.
\end{equation}
In view of Proposition~\ref{theorem:2.8} and Corollary~\ref{theorem:2.9} this yields the proof of Theorem~\ref{theorem:1.2}.
\end{proof}


In view of \cite[Theorem~1.1]{Mueller_2019} and Remark~\ref{theorem:2.17} below
one may have expected that the optimal bound  in \eqref{eq:1.3} in Theorem~\ref{theorem:1.2} would be $2/c_0=0.8346268418\ldots$.
However, the following remark shows that the bound \eqref{eq:2.21} can
presumably not be improved, at least not for supersolutions in the class of ``sufficiently smooth'' graphs.

\begin{remark}\label{theorem:2.16}
	We consider a solution $Y(\cdot)$ as before, parameterised by arclength
	on $[-s_0,s_0]$. The point  $s_0$ is chosen such that $Y(0)-Y(s_0)$ has in modulus
	the maximal slope $-2.3780929080\ldots$ mentioned before. Beyond $s_0$ we extend this curve by the solution of the initial value problem
	\begin{equation*}
	\left\{
	\begin{array}{l}
	\tilde{k}_{ss}+\frac12 \tilde k =0, \quad s\ge s_0\\
	\tilde{Y} (s_0)= Y(s_0), \quad \tilde{Y}_s (s_0)= Y_s(s_0), \quad 
	\tilde{k} (s_0)= k(s_0), \quad \tilde{k}_s (s_0)>> k_s(s_0).
	\end{array}
	\right. 
	\end{equation*}
	Choosing $\tilde{k}_s (s_0)$ arbitrarily large yields an extension 
	$\tilde Y$ which becomes horizontal in $s_1>s_0$ arbitrarily close to 
	$s_0$. We consider this composed symmetric curve on $[-s_1,s_1]$ and 
	rescale and translate it such that it satisfies Dirichlet conditions over
	$[0,1]$. Since the derivative of its curvature jumps upwards in $\{\pm s_0\}$
	we find a smooth enough supersolution of the elastica equation. 
	As a curve it has the regularity as in Proposition~\ref{theorem:2.8} and it is 
	a graph. However, one should observe that it is not even  $C^1$ as a graph. 
	
	So, formally this example does not show the optimality of the bound  \eqref{eq:2.21} for ``smooth'' graphical supersolutions. 
	However, we are confident that  approximating  the delta distributions on the right hand side of the elastica equation by smooth functions would show the optimality also rigorously.
\end{remark}

\begin{remark}\label{theorem:2.16.a}
	Moreover, we believe that one may find obstacles such that the supersolutions 
	constructed in the previous Remark~\ref{theorem:2.16} solve  the corresponding obstacle problems. 
	If this expectation were correct the bound \eqref{eq:1.3} would also be optimal for admissible obstacles.
	However, technically this is a rather demanding project and may be addressed in future work.
	Nevertheless the following reasoning may give some evidence to our belief.
	
	In this remark  we prescribe a suitable boundary slope
	$$
	u'(0)=-u'(1)=\beta \quad \mbox{\ with\ } \quad \beta \in (0,\infty),
	$$
	i.e., we work in the class 
	\begin{equation}\label{eq:2.16.a.1}
	\begin{array}{r}
	\mathcal{M}_\beta (\psi) := \{ u \in  H^2(0,1)  \mid  u(0)= u(1)=0,\,\,\, u'(0)=-u'(1)=\beta,  \\
	u(x) = u(1-x), \,\,\, u(x) \ge \psi(x), \,\,\, \text{for all} \,\,\, x \in (0,1) \}.
	\end{array}
	\end{equation}
	We outline how we may construct suitable obstacles $\psi_\beta $ with
	\begin{equation}\label{eq:2.16.a.2}
	\sup_{\beta\in (0,\infty)}\psi_\beta (1/2)= \frac{z_0}{2}=1.1890464540\ldots
	\end{equation}
	and find $u \in \mathcal{M}_\beta(\psi_\beta)$ such that 
	\begin{align*}
	W(u) = \inf_{v \in \mathcal{M}_\beta(\psi_\beta)} W(v). 
	\end{align*}
	
	We take any $\beta \in (0,\infty)$ and $c<c_0$ but close enough to $c_0$ and find then 
	$$
	x_1\in (0,1/2):\quad \frac{c}{2}-cx_1 =G(\beta) \quad \Leftrightarrow\quad x_1=\frac12 -\frac1c G(\beta).
	$$
	Recalling the function $U_c$ from Remark~\ref{theorem:2.4} we shall rescale and translate $U_c|_{[-x_1,1+x_1]}$
	to the interval $[0,1]$:
	$$
	\widetilde{u}_{\beta,c}:[0,1]\to[0,\infty),\quad \widetilde{u}_{\beta,c}(x):= \frac{1}{1+2x_1}\Big(  U_c \big(-x_1+(1+2x_1)x\big) + U_c(x_1)\Big).
	$$
	Since $U_c(\, .\,)$ is even around $1/2$ the same holds true for $\widetilde{u}_{\beta,c}$. Moreover
	$$
	\widetilde{u}_{\beta,c} (0)=0,\quad \widetilde{u}_{\beta,c}' (0)=U_c' (-x_1)=u_c'(x_1)=G^{-1}\left( \frac{c}{2}-cx_1\right) =\beta.
	$$
	Defining 
	$$
	\psi_{\beta,c,\varepsilon}:[0,1] \to \mathbb{R},\quad \psi_{\beta,c,\varepsilon}(x) := \widetilde{u}_{\beta,c}(x)-\varepsilon
	$$
	for $\varepsilon>0$ but very close to $0$, we find that $\psi_{\beta,c,\varepsilon}$ satisfies Condition~\eqref{eq:A} and that
	$$
	\widetilde{u}_{\beta,c}\in \mathcal{M}_\beta \left(\psi_{\beta,c,\varepsilon} \right).
	$$
	As for the elastic energy we calculate by making use of Lemma~\ref{theorem:2.1}
	and the choice of $x_1$:
	\begin{align*}
	W(\widetilde{u}_{\beta,c})=&(1+2x_1)\cdot\big( W(u_c|_{[0,1]}) +2\cdot W(u_c|_{[0,x_1]})\big)\\
	=& (1+2x_1)\cdot \big( c^2 +2c^2x_1\big)= (1+2x_1)^2c^2=4 (c-G(\beta))^2.
	\end{align*}
	This shows that for the obstacle $\psi_{\beta,c,\varepsilon}$ described above we have that
	$$
	\inf_{v \in \mathcal{M}_\beta(\psi_{\beta,c,\varepsilon})} W(v)\le 4 (c-G(\beta))^2 \mbox{\ with equality iff\ } \widetilde{u}_{\beta,c}
	\mbox{\ is minimal.}
	$$
	This implies that when investigating minimising sequences it suffices to consider
	$$
	v\in \mathcal{M}_\beta(\psi_{\beta,c,\varepsilon}) \mbox{\ with\ } W(v)\le 4 (c-G(\beta))^2.
	$$
	For such $v$ we mimick the proof of Lemma~\ref{theorem:2.2} and pick $x_{\rm max} \in (0,1)$  such that 
	$
	v'(x_{\rm max}) = \max_{x \in [0, 1]}| v'(x) |. 
	$
	For simplicity we consider only the case 
	where $x_{\rm max} \in (0, 1/2)$. 
	By H\"older's inequality, we have 
	\begin{align*}
	& 4 G(v'(x_{\rm max})) -2G(\beta)\\
	&= \int^{v'(x_{\rm max})}_{v'(0)} \dfrac{d \tau}{(1+ \tau^2)^{\frac{5}{4}}} +  \int^{v'(x_{\rm max})}_{v'(1-x_{\rm max})} \dfrac{d \tau}{(1+ \tau^2)^{\frac{5}{4}}} 
	+ \int^{v'(1)}_{v'(1-x_{\rm max})} \dfrac{d \tau}{(1+ \tau^2)^{\frac{5}{4}}} \\
	&= \int^{x_{\rm max}}_{0} \dfrac{v''(x)}{(1 + v'(x)^2)^{\frac{5}{4}}} \, dx 
	+ \int^{1-x_{\rm max}}_{x_{\rm max}} \dfrac{- v''(x)}{(1 + v'(x)^2)^{\frac{5}{4}}} \, dx 
	+ \int^{1}_{1-x_{\rm max}} \dfrac{v''(x)}{(1 + v'(x)^2)^{\frac{5}{4}}} \, dx \\
	& \le \int^1_0 |\kappa_v(x)| (1 + v'(x)^2)^{1/4} \, dx 
	\le W(v)^{1/2} \le 2 (c-G(\beta)).
	\end{align*}
	We obtain as in  the proof of Lemma~\ref{theorem:2.2} that 
	\begin{align} 
	\max_{x \in [0, 1]}| v'(x) |=	v'(x_{\rm max}) \le G^{-1}\left( \dfrac{c}{2} \right) < \infty. 
	\end{align}
	Basing upon this estimate we find as in the proof of Theorem~\ref{theorem:1.1} some
	$$
	u\in  \mathcal{M}_\beta(\psi_{\beta,c,\varepsilon}) \mbox{\ \ with\ \ }
	W(u) = \inf_{v \in \mathcal{M}_\beta(\psi_{\beta,c,\varepsilon})} W(v).
	$$
	In order to show also \eqref{eq:2.16.a.2} we finally observe that
	\begin{align*}
	\lim_{k\to\infty} \psi_{z_0,c_0-1/k,1/k}\left(\frac12\right)=&\widetilde{u}_{z_0,c_0}\left(\frac12\right)
	=\frac{1}{1+2x_1} \left(U_{c_0}\left(\frac12\right)+U_{c_0} (x_1)\right)\\
	=&\frac{1}{1+2x_1} \left(u_{c_0}\left(\frac12\right)+u_{c_0} (x_1)\right)\\
	=&\frac{2}{c_0(1+2x_1)}\left( 1+\frac{1}{(1+G^{-1} (c_0/2-c_0x_1)^2)^{1/4}}\right) \\
	=&\frac{1}{c_0-G(z_0)}\left( 1+\frac{1}{(1+z_0^2)^{1/4}}\right) =\frac{z_0}{2}.
	\end{align*}
	In the last step we used \eqref{eq:2.16.a.3}.
\end{remark}

\begin{remark}\label{theorem:2.17}
	The same reasoning leading to \eqref{eq:2.21} could also be applied 
	to the Navier problem where in arclength parameterisation 
	$\kappa(L)=0$. Using a comparison solution as before with $k(0)=\kappa(0)$,
	$y(0)=Y(0)$, $\tau (0)=T(0)$ would lead to $k(L)\le 0$, i.e. there one is still
	in the concave regime, where the average slope is $\le 4/c_0$. 
	As before, $y$ is right and above $Y$ so that the average slope of $u$ is at most $4/c_0$. 
	This shows that any sufficiently smooth supersolution over $[0,1]$,
	which is symmetric around $1/2$ and satisfies Navier boundary conditions, obeys the universal bound
\begin{equation*}
\forall x\in [0,1]:\quad 
0\le u(x) \le \frac{2}{c_0}.
\end{equation*}	
This is a significant generalisation of \cite[Lemma~4.3]{DD} and \cite[Theorem~1.1]{Mueller_2019}.
\end{remark}

\subsection{The obstacle problem in the non-symmetric case} \label{subsection:2.5}
Let $\psi$ be an obstacle function, which needs no longer to be symmetric and  is always subject 
to the following condition: 
\begin{align} \label{eq:Astar} \tag{${\rm A}_*$}
\psi \in C^0([0,1]), \quad 
\psi (0)<0,\quad \psi(1)<0,\quad 
\exists x_0\in (0,1): \psi (x_0) >0.
\end{align}

We define the  set of admissible functions as follows: 
\begin{align*}
\mathcal{N}(\psi) := \{ v \in H^2_0(0,1) \mid v(x) \ge \psi(x) \,\, \text{for all}\,\, x \in [0,1] \}. 
\end{align*}
Assume that $\psi$ satisfies 
\begin{align} \label{eq:C*} \tag{${\rm B}_*$}
\inf_{v \in N(\psi)} W(v) < c_0^2. 
\end{align}

\begin{lemma} \label{theorem:2.18}
Assume that $v \in \mathcal{N}(\psi)$ satisfies $W(v) \le c_2^2$ for some $c_2\in [0,c_0)$. Then 
\begin{align*}
\max_{x \in [0,1]} |v'(x)| \le G^{-1}\left( \dfrac{c_2}{2} \right) < \infty.
\end{align*}
\end{lemma}
\begin{proof}
Let $x_{\rm max} \in (0,1)$ be such that $|v'(x_{\rm max})| = \max_{x \in [0,1]}|v'(x)|$. 
First we consider the case where $v'(x_{\rm max})>0$. 
Then there exists $x_0 \in (x_{\rm max},1)$ where $v$ attains its maximum which means that
\begin{align*}
v'(x_0)=0, \quad v''(x_0)\le 0. 
\end{align*}
Considering $v |_{[0, x_0]}$, we can adopt the argument from the proof of Lemma \ref{theorem:2.2}. 
Then we have 
\begin{align*}
2 G(v'(x_{\max})) 
 & = \int^{v'(x_{\max} )}_0 \dfrac{d \tau}{(1+ \tau^2)^{5/4}} + \int^{v'(x_{\max} )}_0 \dfrac{d \tau}{(1+ \tau^2)^{5/4}} \\
 &= \int^{x_{\rm max}}_{0} \dfrac{v''(x)}{(1 + v'(x)^2)^{5/4}} \, dx + \int^{x_0}_{x_{\rm max}} \dfrac{-v''(x)}{(1 + v'(x)^2)^{5/4}} \, dx \\
 &\le W(v)^{1/2} \le c_2. 
\end{align*}
Thanks to the monotonicity of of $G$, we obtain 
\begin{align*}
v'(x_{\rm max}) \le G^{-1}\left( \dfrac{c_2}{2} \right). 
\end{align*}

For the case where $v'(x_{\rm max}) < 0$, considering $-v$ instead of $v$ and using the same argument as in the first case,  
we obtain the required conclusion. The proof is complete.  
\end{proof}

Combining Lemma \ref{theorem:2.18} with the same argument as in the proof of Theorem \ref{theorem:1.1} and Proposition \ref{theorem:2.8}, we obtain the following:  
\begin{theorem} \label{theorem:2.19}
Assume that $\psi\in C^0([0,1])$ satisfies conditions~\eqref{eq:Astar} and \eqref{eq:C*}. 
Then there exists $u \in \mathcal{N}(\psi)$ such that 
\begin{align*}
W(u) = \inf_{v \in \mathcal{N}(\psi)} W(v). 
\end{align*}
Moreover, $u \in C^2([0,1])$, $u''$ is weakly differentiable and $u''' \in BV(0,1)$. 
\end{theorem}

\begin{remark}\label{theorem:2.20}{\rm 
Analogues of Propositions~\ref{theorem:2.10} and \ref{theorem:2.11} can be proved also in the nonsymmetric case.
}
\end{remark}

\begin{remark}\label{theorem:2.21}{\rm \ 
We use the same notation as in Remark~\ref{theorem:2.4} and consider the functions 
$\hat{u}_c : [0, 1] \to \mathbb{R}$ as there. We recall that
\begin{align*}
W(\hat{u}_c) = 2 W(U_c) = 4 c^2. 
\end{align*}
In order to satisfy condition~\eqref{eq:C*} we have now to assume that
$$c \in (0, c_0/2),\quad \frac{c_0}{2} =
\sqrt{\pi} \frac{\Gamma(3/4)}{2\Gamma(5/4)}=1.198140234\ldots,
$$
here we have that $W(\hat{u}_c)< c_0^2$. 
This implies that for any sufficiently small $\varepsilon >0$, any $\psi \in C^0([0,1])$ with $\psi \le \hat{u}_c-\varepsilon $ obeys conditions~\eqref{eq:Astar} and~\eqref{eq:C*}. 
In particular, any $\hat{u}_c-\varepsilon$  with $c\in (0, c_0/2)$ is itself an admissible obstacle.

\begin{figure}[h] 
\centering 
\includegraphics[width=.45\textwidth]{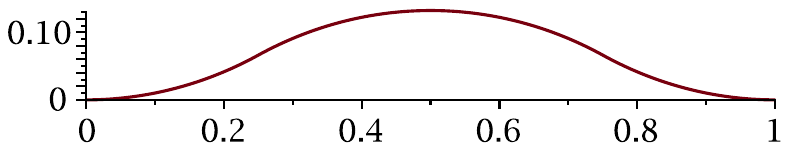}\hspace{0.3cm}
\includegraphics[width=.45\textwidth]{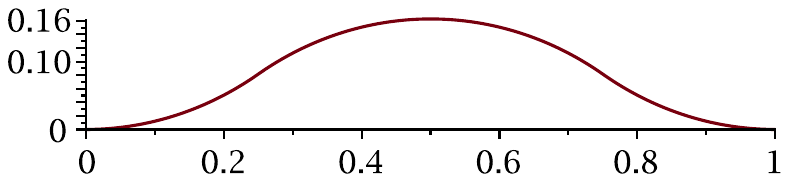}\vspace{-5cm}
\caption{$\hat{u}_{1.0}$  (left) and $\hat{u}_{c_0/2}$  (right)}
\label{figure:2} 
\end{figure} 

One may observe that
$$
\max_{x\in[0,1]} \hat{u}_{c_0/2}=\hat{u}_{c_0/2} (1/2)
=\dfrac{4}{c_0}\left(1- \dfrac{1}{ \sqrt[4]{1 + G^{-1}\left( \tfrac{c_0}{4}\right)^2}}\right)
=0.1628208198\ldots.
$$
}
\end{remark}


\section{An obstacle problem for surfaces of revolution} \label{section:3}
Let $u : [-1, 1] \to (0, \infty)$ be a profile curve with Dirichlet boundary conditions 
\begin{align}\label{eq:3.1}
u(1)=u(-1)= \alpha > 0, \quad u'(1)=u'(-1)=0. 
\end{align}
Then the mean curvature $H$ and Gauss curvature $K$ of the surface of revolution which are obtained by 
rotating the graph of $u$ around the $x$-axis 
$$
[-1,1]\times [0,2\pi]\ni (x,\varphi)\mapsto(x,u(x)\cos\varphi,u(x)\sin\varphi)
$$
are given as follows: 
\begin{align*}
H &= \dfrac{1}{u(x) \sqrt{1+(u'(x))^2}} - \dfrac{u''(x)}{(1 + (u'(x))^2)^{3/2}}, \\
K &= -\dfrac{u''(x)}{u(x)(1 + (u'(x))^2)^2}.  
\end{align*}
Thus the Willmore functional is given by 
\begin{equation}
\label{eq:3.2}
\begin{aligned}
W(u) 
&= \dfrac{\pi}{2} \int^1_{-1} \Bigl( \dfrac{1}{u(x) \sqrt{1+(u'(x))^2}} - \dfrac{u''(x)}{(1+(u'(x))^2)^{\frac{3}{2}}} \Bigr)^2 u(x) \sqrt{1 + (u'(x))^2} \, dx \\
&= \dfrac{\pi}{2} \int^1_{-1} \dfrac{u''(x)^2u(x)}{(1+(u'(x))^2)^{\frac{5}{2}}}\, dx
   +\dfrac{\pi}{2} \int^1_{-1} \dfrac{1}{u(x) \sqrt{1+(u'(x))^2}}\, dx \\
 &\quad  -  \pi \Bigm[ \dfrac{u'(x)}{\sqrt{1+u'(x)^2}} \Bigm]^{1}_{-1} \\
&= \dfrac{\pi}{2} \int^1_{-1} \dfrac{u''(x)^2u(x)}{(1+(u'(x))^2)^{\frac{5}{2}}}\, dx
+\dfrac{\pi}{2} \int^1_{-1} \dfrac{1}{u(x) \sqrt{1+(u'(x))^2}}\, dx.
\end{aligned}
\end{equation}
Bryant and Griffiths (see \cite{BryantGriffiths_1986}) and independently  Pinkall (cf. e.g. \cite{Herterich_Pinkall}) made an important observation concerning the Willmore energy $W(u)$ for surfaces of revolution and the elastic energy of the graph of $u$, considered as a curve in the hyperbolic space. This is the upper half plane $\mathbb{R}\times (0,\infty)$ equipped with  the metric 
$ds^2=\frac1{y^2}(dx^2+dy^2)$. The hyperbolic curvature and the hyperbolic elastica functional of the graph of $u$ are given by the formulae: 
\begin{align*}
\kappa_h(x)&= \dfrac{u(x) u''(x)}{(1+(u'(x))^2)^{3/2}} + \dfrac{1}{\sqrt{1+(u'(x))^2}}, \\
W^h(u)&= \int^1_{-1} \kappa_h(x)^2 \dfrac{\sqrt{1+(u'(x))^2}}{u(x)} \, dx .
\end{align*}
It follows then  from a simple calculation that 
\begin{align} \label{eq:3.3}
W(u|_{[a,b]}) = \dfrac{\pi}{2} W^h(u|_{[a,b]}) - 2 \pi \Bigm[ \dfrac{u'(x)}{\sqrt{1+u'(x)^2}} \Bigm]^{a}_{b}, 
\end{align}
see e.g. \cite[p. 384]{DDG} or \cite[p. 378]{GGS}. 
This relation turned out to be very useful and was intensively exploited by Langer and Singer (see \cite{LangerSinger_1984a,LangerSinger_1984b}). Moreover, among others, they uncovered remarkable differential geometric properties 
of hyperbolic elasticae which were subsequently exploited also by Eichmann and Grunau (see e.g. \cite{Eichmann2016,EichGr}). A comprehensive exposition of the analysis and differential geometry of the Dirichlet problem of the (euclidean as well as hyperbolic) elastica equation can be found in \cite{Mandel1, Mandel2}.

\subsection{Willmore surfaces of revolution under Dirichlet boundary conditions}
\label{subsection:3.1} 

As in the one-dimensional case it is good to have results about the existence, the regularity  and the 
shape of minimisers  (without obstacle) under Dirichlet boundary conditions \eqref{eq:3.1} in  mind.

Combining \cite[Theorem 1.1]{DDG}, \cite[Lemma 3.20]{DFGS}, \cite[Theorem 1.1]{EK}, one has the following result:
\begin{theorem}\label{theorem:3.1}
For each $\alpha>0$ and 
there exists a function $u\in C^\infty([-1,1],(0,\infty))$ which minimises 
the Willmore energy $W(\, .\, )$ in the class
$$
\{ v\in H^2((-1,1);(0,\infty))\, | \, v(1)=v(-1)= \alpha , \quad v'(1)=v'(-1)=0 \}.
$$
The corresponding surface of revolution $\Gamma\subset\mathbb R^3$ solves the Dirichlet problem 
\begin{equation*}
\left\{
      \begin{array}{l}
        \triangle_\Gamma H+2H(H^2-K)=0\quad\mbox{in}\ (-1,1) ,\\[1ex]
        u(-1)=u(+1)=\alpha,\qquad
        u'(-1)=-u'(+1)=0.
      \end{array}
    \right.
\end{equation*}
Any such minimiser is symmetric  (i.e. $u(x)=u(-x)$) and  has the following properties: 
\begin{align*}
-\frac{x}{\alpha}< u'(x)< 0, \quad x+u(x)u'(x) > 0, \quad \text{in} \quad (0,1), 
\end{align*}
and 
\begin{align*}
\alpha < u(x) < \sqrt{1+\alpha^2-x^2} \quad \text{in} \quad (-1,1). 
\end{align*}
\end{theorem}

For plots of numerically calculated solutions one may see \cite[Figure 17]{DFGS}.
For $\alpha\searrow 0$ one finds boundary layers close to $\pm 1$ of width $\alpha$ where any minimiser 
approaches after rescaling by the factor $1/\alpha$ a catenoid.
On the other hand,  in any compact subset of $(-1,1)$ 
any minimiser approaches the upper unit half circle centered at $(0, 0)$, 
see \cite{G}.

\subsection{The obstacle problem} 
\label{subsection:3.2}

Assume that $\psi \in C^0([-1,1]; (0, \infty))$ satisfies condition \eqref{eq:C}. 
We recall the definition in \eqref{eq:1.5.a} of  the admissible set $N_\alpha(\psi)$. 
In the following,  $M_\alpha(\psi)$ denotes the infimum of the Willmore functional $W(v)$ in this admissible set, i.e., 
\begin{equation*}
M_\alpha(\psi) := \inf_{v \in N_\alpha(\psi)} W(v). 
\end{equation*}

\begin{lemma} \label{theorem:3.2}
Assume that $v \in N_\alpha(\psi)$. If
\begin{align} \label{eq:3.4}
W(v) < 4 \pi
\end{align}
is satisfied then it holds for all $x \in [-1,1]$ that 
\begin{equation}
|v'(x)| \le K:=  \dfrac{1}{\sqrt{\left( \tfrac{4\pi}{W(v)} \right)^2 - 1}}\ . \label{eq:3.5}
\end{equation}
If for some $K$, $|v'(x)| \le K$ is satisfied on $[-1,1]$, we have 
\begin{equation}
\alpha + K \ge v(x) \ge M:= \dfrac{K}{\exp(\frac{2K}{\pi}\sqrt{1+K^2\, }\,\, W(v))-1} \label{eq:3.6}
\end{equation}
for all $x \in [-1,1]$.
\end{lemma}
\begin{proof}
To begin with, we prove \eqref{eq:3.5}. 
Let $x_{\rm max} \in (0,1)$ be such that 
\begin{align*}
|v'(x_{\rm max}) | = \max_{x \in [-1,1]} |v'(x)| =: K_1. 
\end{align*}
First we consider the case where $v'(x_{\rm max}) < 0$. 
Since $v \in N_\alpha(\psi)$, we find 
\begin{align} \label{eq:3.7}
v'(-x_{\rm max}) = - v'(x_{\rm max}) > 0.  
\end{align} 
It follows from \eqref{eq:3.3} and \eqref{eq:3.7} that 
\begin{align*}
W(v) &= W(v|_{[-1,-x_{\rm max}]}) + W(v|_{[-x_{\rm max},x_{\rm max}]}) + W(v|_{[x_{\rm max},1]}) \\
 & \ge \dfrac{\pi}{2} W^h(v|_{[-x_{\rm max}, x_{\rm max}]}) - 2 \pi \Bigm[ \dfrac{v'(x)}{\sqrt{1+v'(x)^2}} \Bigm]^{x_{\rm max}}_{-x_{\rm max}} 
   \ge 4 \pi \dfrac{K_1}{\sqrt{1+K_1^2}},  
\end{align*}
and then we have
\begin{align} \label{eq:3.8}
\left( \dfrac{W(v)}{4 \pi} \right)^2 \ge 1 - \dfrac{1}{1+K_1^2}. 
\end{align}
By \eqref{eq:3.4} we deduce from \eqref{eq:3.8} that $K_1 \le K$.  
We turn to the case where $v'(x_{\rm max}) > 0$. In this case, we have 
\begin{align} \label{eq:3.9}
v'(-x_{\rm max}) = - v'(x_{\rm max}) < 0.  
\end{align} 
Using \eqref{eq:3.3} again, we obtain 
\begin{align*}
W(v) &= W(v|_{[-1,-x_{\rm max}]}) + W(v|_{[-x_{\rm max},x_{\rm max}]}) + W(v|_{[x_{\rm max},1]}) \\
 & \ge W(v|_{[-1,-x_{\rm max}]}) + W(v|_{[x_{\rm max},1]}) \\ 
 & \ge \dfrac{\pi}{2} W^h(v|_{[-1,-x_{\rm max}]}) - 2 \pi \Bigm[ \dfrac{v'(x)}{\sqrt{1+v'(x)^2}} \Bigm]^{-x_{\rm max}}_{-1} \\
 & \quad + \dfrac{\pi}{2} W^h(v|_{[x_{\rm max}, 1]}) - 2 \pi \Bigm[ \dfrac{v'(x)}{\sqrt{1+v'(x)^2}} \Bigm]^{1}_{x_{\rm max}} \\
 & \ge 2 \pi \left( \dfrac{v'(x_{\rm max})}{\sqrt{1+ v'(x_{\rm max})^2}} - \dfrac{v'(-x_{\rm max})}{\sqrt{1+ v'(-x_{\rm max})^2}} \right),  
\end{align*}
This together with \eqref{eq:3.9} implies that 
\begin{align*}
W(v) \ge 4 \pi \dfrac{K_1}{\sqrt{1+K_1^2}}. 
\end{align*}
Then we find $K_1 \le K$, as in the first case. 

We prove \eqref{eq:3.6}. 
Thanks to \eqref{eq:3.5}, we observe that for $x\in [-1,0]$
\begin{align*}
v(x) = \alpha + \int^x_{-1} v'(\xi) \, d \xi \le \alpha + \int^0_{-1} |v'(\xi)| \, d \xi \le \alpha +  K. 
\end{align*}
Since $v$ is even, the estimate from above follows.

For the estimate from below, we let $x_{\rm min} \in [0,1]$ be such that 
\begin{align*}
v(x_{\rm min}) = \min_{x \in [-1,1]} v(x) =: v_{\rm min} > 0. 
\end{align*}
Then, for any $x \in [-1, x_{\rm min}]$, we have 
\begin{align} \label{eq:3.10}
v(x) = v_{\rm min} - \int^{x_{\rm min}}_x v'(\xi) \, d \xi \le v_{\rm min} + K(x_{\rm min}-x).
\end{align}
Thanks to $v'(-1)=v'(1)=0$, we obtain 
\begin{equation} 
\label{eq:3.11}
\begin{aligned}
W(v) &= \dfrac{\pi}{2} W^h(v) \\
     &= \dfrac{\pi}{2} \int^1_{-1} \dfrac{v (v'')^2}{(1+(v')^2)^{5/2}} \, dx + \pi \int^1_{-1} \dfrac{v''}{(1+(v')^2)^{3/2}} \, dx \\
     & \qquad + \dfrac{\pi}{2} \int^1_{-1} \dfrac{1}{v \sqrt{1+(v')^2}} \, dx \\
     &\ge \pi \Bigm[ \dfrac{v'(x)}{\sqrt{1+ v'(x)^2}} \Bigm]^1_{-1} + \dfrac{\pi}{2} \int^1_{-1} \dfrac{1}{v \sqrt{1+(v')^2}} \, dx \\
     &= \dfrac{\pi}{2} \int^1_{-1} \dfrac{1}{v \sqrt{1+(v')^2}} \, dx. 
\end{aligned}
\end{equation}
Combining \eqref{eq:3.11} with \eqref{eq:3.5} and \eqref{eq:3.10}, we deduce that 
\begin{equation} \label{eq:3.12}
\begin{split}
W(v) &\ge \dfrac{\pi}{2} \dfrac{1}{\sqrt{1+K^2}} \int^{x_{\rm min}}_{x_{\rm min}-1} \dfrac{1}{v}\, dx \\
     &\ge \dfrac{\pi}{2} \dfrac{1}{\sqrt{1+K^2}} \int^{x_{\rm min}}_{x_{\rm min}-1} \dfrac{1}{v_{\rm min} + K(x_{\rm min}-x)}\, dx \\
     &\ge \dfrac{\pi}{2} \dfrac{1}{\sqrt{1+K^2}} \int^{1}_{0} \dfrac{1}{v_{\rm min} + K \xi}\, d\xi 
      =  \dfrac{\pi}{2K\sqrt{1+K^2}} \log{\dfrac{v_{\rm min} + K}{v_{\rm min}}}.
\end{split}
\end{equation}
By a direct calculation, we observe from \eqref{eq:3.12} that 
\begin{align*}
v_{\rm min} \ge M. 
\end{align*}
The proof is complete. 
\end{proof}

We may now prove our first existence result for surfaces of revolution.

\begin{proof}[Proof of Theorem \ref{theorem:1.3}]
Let 
\begin{align*}
\tilde{M}_\alpha(\psi) := \frac{1}{2}(4 \pi + M_\alpha(\psi)) \in ( M_\alpha(\psi), 4 \pi),  
\end{align*}
and consider a minimising sequence $\{ v_k \}_{k \in \mathbb{N}} \subset N_\alpha(\psi)$. 
We may assume that 
\begin{align} \label{eq:3.13}
M_\alpha (\psi) \le W(v_k) \le \tilde{M}_\alpha(\psi). 
\end{align}
By Lemma \ref{theorem:3.2} we see that $\{ v_k \}_{k \in \mathbb{N}}$ is uniformly 
$C^1$-bounded, to be more precisely, 
\begin{gather*}
|v'_k(x)| \le \tilde{K} := \dfrac{\tilde{M}_\alpha(\psi)}{4 \pi} \dfrac{1}{\sqrt{1-\left( \tfrac{\tilde{M}_\alpha(\psi)}{4\pi} \right)^2}}, \\
\alpha +  \tilde{K} \ge v_k(x) \ge \tilde{M}:= \dfrac{\tilde{K}}{\exp(\frac{2 \tilde{K}}{\pi}\sqrt{1+ \tilde{K}^2} \tilde{M}_\alpha(\psi) )-1}, 
\end{gather*} 
for all $x \in [-1,1]$ and $k \in \mathbb{N}$. This together with \eqref{eq:3.13} implies that $\{ v_k \}_{k \in \mathbb{N}}$ is uniformly bounded in $H^2$. 
Then, analogously to the proof of Theorem~\ref{theorem:1.1}, we find $u \in N_\alpha(\psi)$ as the limit of $\{ v_k \}_{k \in \mathbb{N}}$ such that 
\begin{align*}
W(u) = \min_{v \in N_\alpha(\psi)} W(v). 
\end{align*}
In particular, the minimiser $u$ satisfies $u(x) \ge \tilde{M}>0$ for all $x \in [-1, 1]$. The proof is complete. 
\end{proof}

\begin{remark}\label{theorem:3.3}
{\rm 
For which $\alpha$ do we easily obtain interesting and admissible obstacles? 
We  construct a function $v_\alpha \in N_\alpha$ with $W(v_\alpha) < 4\pi$. Any $\psi\in C^0([-1,1];(0,\infty))$ with $\psi(\pm1)>\alpha$  
and $\psi\ge v_\alpha $ would then be an admissible obstacle and $v_\alpha\in N_\alpha (\psi)$.

To this end, let $b_0>0$ denote the solution of the equation $b_0\cdot \tanh(b_0)=1$. One sees
that the space above and including  the graph of $x\mapsto \left| \frac{\cosh(b_0)}{b_0}\cdot x\right| $ coincides
with the set $\{(x,\cosh(bx)/b)\, | \, x\in \mathbb{R},\, b>0\} \cup \{(0,0)\}$. The smallest circle around $(1,0)$
which intersects this set has the radius 
$$
\alpha_0:=\sqrt{1-\frac{1}{1+\left(\frac{\cosh(b_0)}{b_0}\right)^2}}=0.8335565596\ldots,\ \mbox{\ where\ }
\frac{\cosh(b_0)}{b_0}=1.508879561\ldots.
$$
One may observe that $\alpha_0$ is close to, but slightly below the number $2/c_0$ which plays an important role in Section~\ref{section:2}.

In what follows we assume  first that $\alpha \ge \alpha_0$. For $\alpha\ge 1$ and $\alpha=\alpha_0$ one finds one 
catenoid which touches the circles around $(1,0)$ and $(-1,0)$ with  radii $\alpha$.
For $\alpha\in (\alpha_0,1)$ one has even  two such catenoids. 

\begin{figure}[h] 
\centering 
\resizebox{10cm}{!}{\includegraphics{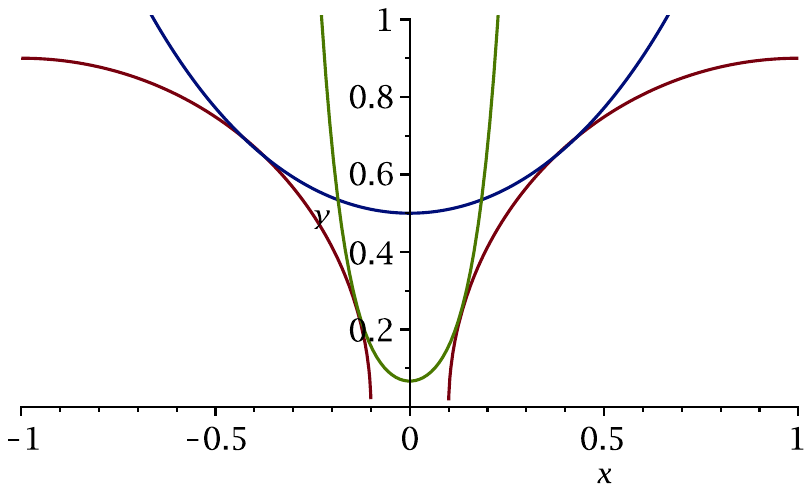}}\vspace{-7cm}
\caption{Any smooth function built from the brown circular arcs and 
a part either of the green or of the blue $\cosh$ 
and slightly enlarged near $\pm 1$ is an admissible obstacle for $\alpha=0.9$.}
\label{figure:3} 
\end{figure} 

To be more explicit,
we put
$$
v_\alpha(x):=\left\{
\begin{array}{ll}
       \sqrt{\alpha^2-(1+x)^2} \quad & \quad \mbox{\ for\ } -1\le x\le -x_b,\\[2mm]
       \frac{1}{b} \cosh(bx) \quad & \quad \mbox{\ for\ }  -x_b\le x\le x_b,\\[2mm]
       \sqrt{\alpha^2-(1-x)^2} \quad & \quad \mbox{\ for\ } x_b\le x\le 1,\\[2mm]
      \end{array}
\right.
$$
where $b$ and $x_b$ are chosen such that $v_\alpha \in H^2((-1,1);(0,\infty))$, i.e.:
\begin{equation*}
 \left\{ \begin{array}{rl}
        \displaystyle  \frac{1}{b} \cosh(bx_b)& =\displaystyle \sqrt{\alpha^2-(1-x_b)^2}\\[2mm]
        \displaystyle \sinh (bx_b) & =\displaystyle  \frac{1-x_b}{\sqrt{\alpha^2-(1-x_b)^2}}.
         \end{array}
 \right. 
\end{equation*}
After some elementary calculations one finds that this is equivalent to solving first
\begin{equation*}
 \cosh(b-\sqrt{(\alpha b)^2-(\alpha b)})=\sqrt{\alpha b}
\end{equation*}
for $b$ and putting then:
\begin{equation*}
 x_b:= 1-\sqrt{\alpha^2 - \frac{\alpha}{b}}.
\end{equation*}
As in the proof of Lemma~\ref{theorem:3.2} one has that
\begin{equation}\label{eq:3.12.a}
W(v_\alpha)=4\pi \tanh(bx_b) <4\pi.
\end{equation}

Some numerically calculated examples:
$$
\begin{array}{|c|c|c|c|} \hline 
 \alpha & b & x_b & v_\alpha (0)=1/b\\ \hline
 5 & 0.2020339962\ldots & 0.4983130059\ldots & 4.949662031\ldots \\ \hline
 2 & 0.5349618217\ldots & 0.4887119667\ldots & 1.869292274\ldots \\ \hline
 1 & 1.467396505\ldots & 0.4356234114\ldots & 0.6814790662\ldots \\ \hline
  0.99 &  1.502200407\ldots & 0.4333724672\ldots & 0.6656901405\ldots \\ \hline
 0.99 &  304.6450597\ldots & 0.0116426170\ldots &  0.003282508507\ldots \\ \hline
 0.9 &  1.986626006\ldots & 0.4025298387\ldots & 0.5033660070\ldots \\ \hline
 0.9 &  14.46598282\ldots & 0.1352543279\ldots &  0.06912769166\ldots \\ \hline
  0.84 &  3.077899286\ldots & 0.3422108342\ldots & 0.3248969206\ldots \\ \hline
 0.84 &  5.266858981\ldots & 0.2610059859\ldots &  0.1898664847\ldots \\ \hline
\end{array}
$$
As mentioned before, for $\alpha\in (\alpha_0,1)$ we find two catenoids which touch the circles around $(1,0)$ and $(-1,0)$ with 
radii $\alpha$. For $\alpha=\alpha_0$ and $\alpha \ge 1$ exactly one catenoid fits in like this.
For $\alpha\searrow \alpha_0$ the two constructed branches of ``admissible obstacles'' (up to being slightly enlarged near $\pm 1$) converge to 
the same one while for $\alpha \nearrow 1$ one of these branches persists (and becomes less and
less interesting for increasing $\alpha$ and the other one becomes singular (convergence to two 
quarter circles). See Figures~\ref{figure:4}--\ref{figure:6}.

\begin{figure}[h] 
\centering 
\includegraphics[width=.45\textwidth]{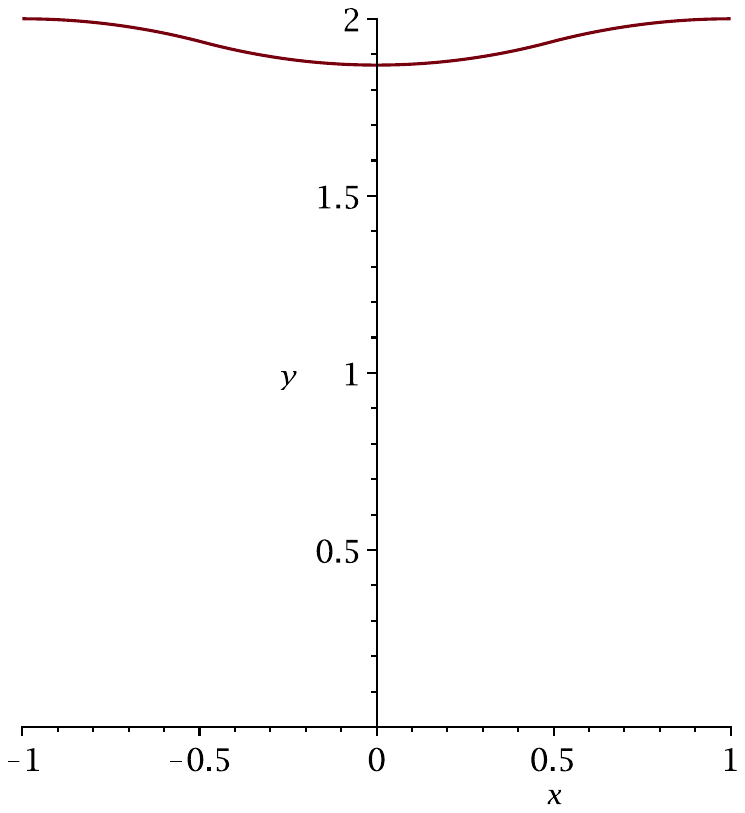}\hspace{0.3cm}
\includegraphics[width=.45\textwidth]{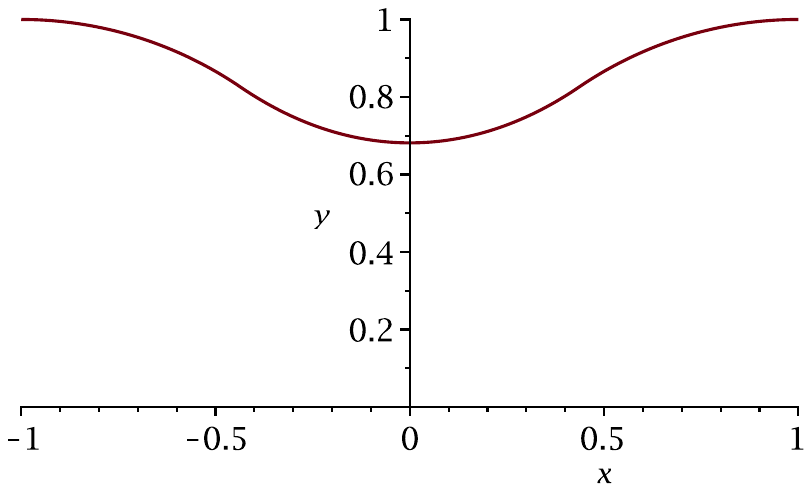}\vspace{-3cm}
\caption{Admissible obstacles (up to enlarging near $\pm 1$)
	 for $\alpha=2$ (left) and $\alpha=1$ (right).}
\label{figure:4} 
\end{figure} 

\begin{figure}[h] 
\centering 
\includegraphics[width=.45\textwidth]{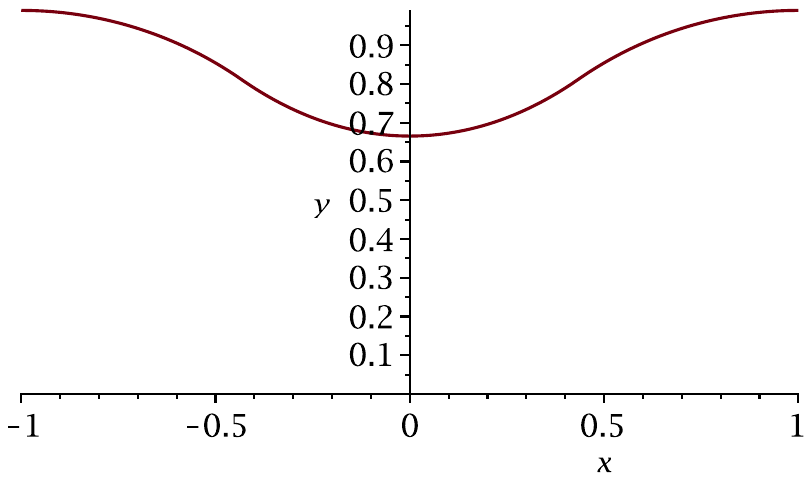}\hspace{0.3cm}
\includegraphics[width=.45\textwidth]{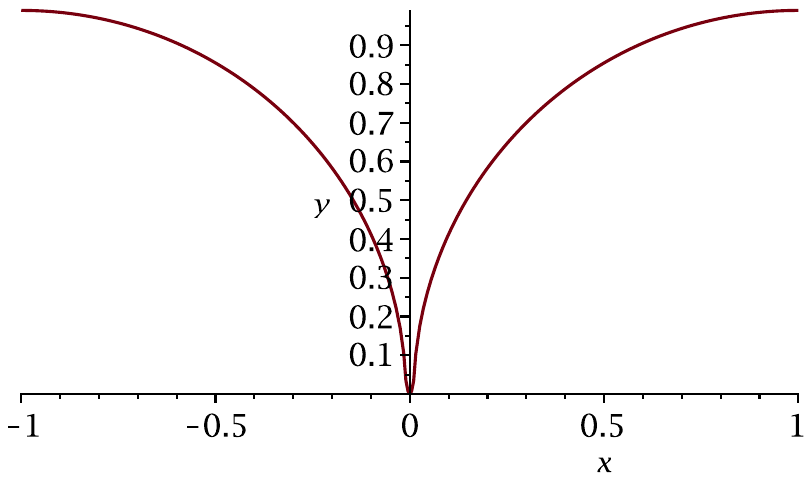}\vspace{-3.5cm}
\caption{Two admissible obstacles  (up to enlarging near $\pm 1$) for $\alpha=0.99$.}
\label{figure:5} 
\end{figure} 

\begin{figure}[h] 
\centering 
\includegraphics[width=.45\textwidth]{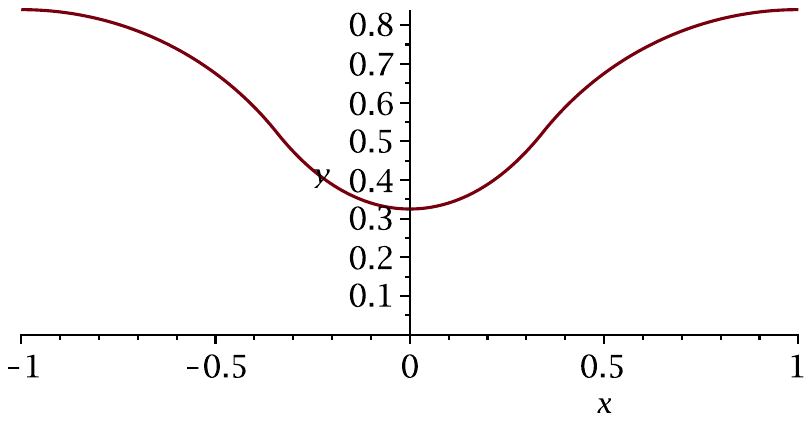}\hspace{0.3cm}
\includegraphics[width=.45\textwidth]{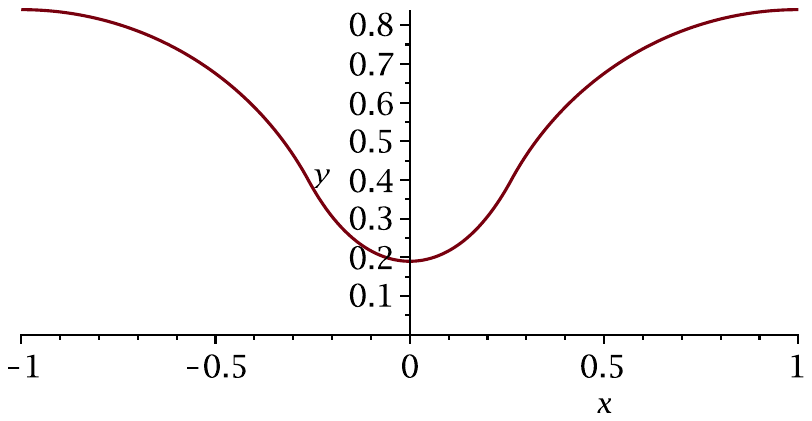}\vspace{-3.5cm}
\caption{Two admissible obstacles  (up to enlarging near $\pm 1$) for $\alpha=0.84$.}
\label{figure:6} 
\end{figure} 

We consider now the case  $\alpha\in (0,\alpha_0)$. Here, the previous  
construction has to be modified as follows: One 
\emph{chooses} suitably (in a sense which will be explained below) $x_b\in (1-\alpha,1)$ such that the catenoids
$$
x\mapsto \frac{\cosh(b(|x|-x_0))}{b}
$$
are in $\pm x_b$ tangential to the circles
$$
x\mapsto \sqrt{\alpha^2-(1-|x|)^2)}.
$$
The parameters $b$ and $x_0$ are given by the equations
\begin{equation*}
b=\frac{\sqrt{1+\beta^2}}{\gamma},\quad \sinh(b(x_b-x_0))=\beta, 
\end{equation*}
where
\begin{equation*}
\gamma:=\sqrt{\alpha^2-(1-x_b)^2}, \quad 
\beta:= \frac{1-x_b}{\sqrt{\alpha^2-(1-x_b)^2}}. 
\end{equation*}
One puts then a circle around the origin which touches these catenoids. Its radius $r$ is 
given by 
$$
r:=\sqrt{x^2_{\min} + \frac{\cosh(b(x_0-x_{\min}))^2}{b^2} }
$$
where $x_{\min}\in (0,x_0)$ is defined as the solution of the equation
$$
b x_{\min}=\cosh(b(x_0-x_{\min}))\sinh(b(x_0-x_{\min})).
$$
One then defines 
$$
v_\alpha(x):=\left\{
\begin{array}{ll}
       \sqrt{r^2-x^2} \quad & \quad \mbox{\ for\ } | x | \le x_{\min},\\[2mm]
       \frac{1}{b} \cosh(b(|x|-x_0)) \quad & \quad \mbox{\ for\ }  x_{\min}\le |x|\le x_b,\\[2mm]
       \sqrt{\alpha^2-(1-|x|)^2} \quad & \quad \mbox{\ for\ } x_b\le |x|\le 1,\\[2mm]
      \end{array}
\right.
$$
Finally, in order that $v_\alpha$ becomes an admissible obstacle  (up to enlarging near $\pm 1$), $x_b$ has to be chosen such that 
$$
W(v_\alpha)= 4\pi \left(  \tanh(b(x_b-x_0))+\tanh(b(x_0- x_{\min}))\right)
\stackrel{!}{<}4\pi.
$$
This is certainly possible because for $x_b\nearrow 1$ the function $v_\alpha$ converges to
the prototype function used in \cite{DFGS}. There, for any $\alpha>0$ it was shown in Proposition 6.6 that 
this function has Willmore energy strictly below $4\pi$. 

On the other hand, for $\alpha \nearrow \alpha_0$, $x_b$ may be chosen such that $x_0$ is still very close 
to $0$ and the admissible obstacles (up to enlarging \ldots) still resemble those from the case $\alpha \searrow \alpha_0$. Roughly speaking the previous construction for 
$\alpha \nearrow \alpha_0$ is as follows: One takes an admissible obstacle
for $\alpha\ge \alpha_0$, $\alpha$ close to $\alpha_0$, scales this down around  $-1$ and $1$ by a factor $<1$ but close to $1$ and fits in a small part of a suitable halfcircle around $0$.

Some numerically calculated parameters for admissible obstacles (up to enlarging \ldots),
 see also Figure~\ref{figure:7}:
$$
\begin{array}{|c|c|c|c|c|c|} \hline 
 \alpha & x_b & b & x_0 & x_{\min} & r\\ \hline
 0.7& 0.5  & 0.4898979486\ldots & 0.1928412335\ldots & 0.09881364931\ldots &  0.3692969430\ldots\\ \hline
0.5& 0.79 & 2.428363283\ldots & 0.6056404249\ldots & 0.3419392371\ldots & 0.6050456522\ldots \\ \hline
0.1& 0.995 & 10.02506266\ldots & 0.9900083375\ldots & 0.8159959886\ldots &  0.8673937881\ldots\\ \hline
\end{array}
$$

\begin{figure}[h] 
\centering 
\includegraphics[width=.45\textwidth]{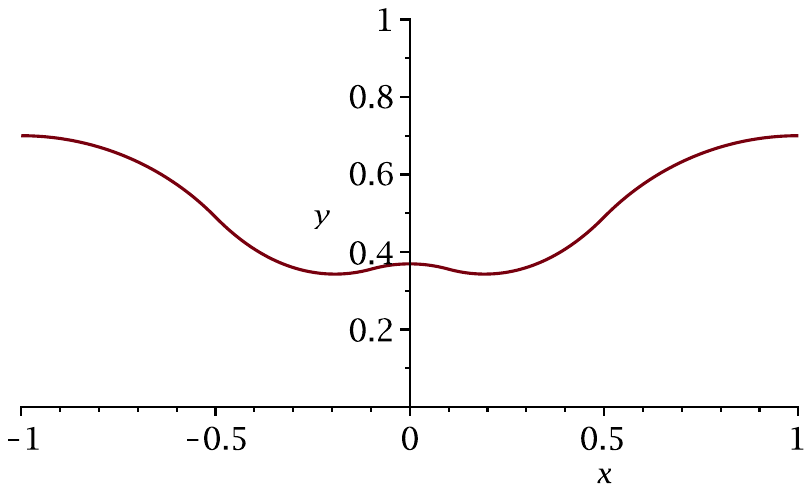}\hspace{0.3cm}
\includegraphics[width=.45\textwidth]{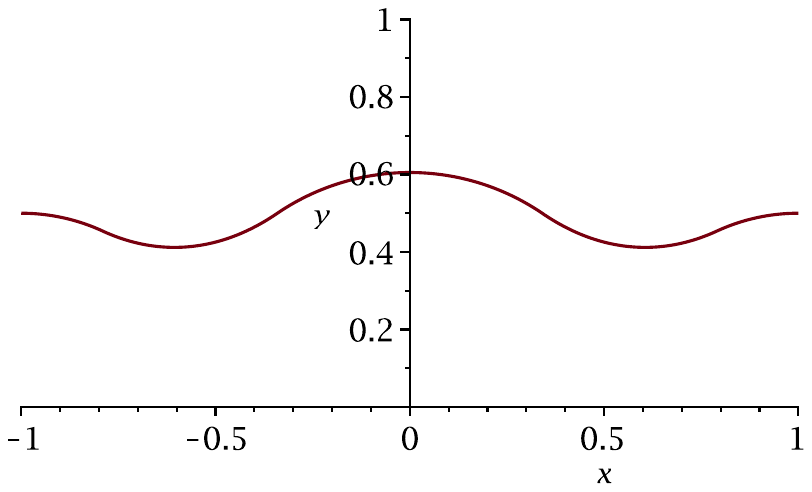}\vspace{-4cm}
\caption{Admissible obstacles (up to enlarging \ldots) for $\alpha=0.7$ (left) and $\alpha=0.5$ (right).}
\label{figure:7} 
\end{figure} 



}
\end{remark}

The previous examples are in our opinion quite interesting for $\alpha$ in a neighbourhood of $1$. But giving a look at Figure~\ref{figure:4} and the table just below \eqref{eq:3.12.a} shows that already
for $\alpha=2$ the constructed ``admissible obstacle'' resembles somehow a straight line. Indeed, it is not difficult to see that $\lim_{\alpha\nearrow \infty}\| \alpha -v_\alpha\|_{C^0([-1,1])}=0$ and $\lim_{\alpha\nearrow \infty}W (v_\alpha )=0$.

The following result shows that the Willmore energy of ``interesting'' examples increases of the order $\alpha$ when $\alpha\nearrow \infty$. On the other hand it can be interpreted in that way that one has suitable a-priori-bounds in $C^1([-1,1],(0,\infty))$ for functions whose energy is below $\alpha$ times a suitable factor. The underlying observation is that in this regime the second term in $W(\,\cdot\,)$ in \eqref{eq:3.2} becomes small compared to the first one and may be neglected. Then, up to a factor we are back in the one-dimensional situation.
In what follows, the function $G$ defined in \eqref{eq:1.1.new}
will play again an important role.
\begin{proposition}\label{theorem:3.4}
Let $S\in (0,\alpha)$. Assume that $u\in N_\alpha$ is such that 
\begin{equation}\label{eq:3.14}
\max_{x\in [-1,1]} |u'(x)| \ge S.
\end{equation}
Then
\begin{equation}\label{eq:3.15}
W(u) > \pi (\alpha -S)G(S)^2.
\end{equation}
Equivalently we may conclude for $u\in N_\alpha$
\begin{equation}\label{eq:3.16}
W(u) \le  \pi (\alpha -S)G(S)^2\quad \Rightarrow \quad 
\max_{x\in [-1,1]} |u'(x)| < S.
\end{equation}
\end{proposition}
\begin{proof} We consider $u\in N_\alpha$ which satisfies \eqref{eq:3.14}.
	Choose $x_0\in (-1,0)$ minimal such that 
	$|u'(x_0)|=S$, i.e. $\forall x\in [-1,x_0): \quad |u'(x)|<S $
	and   $\forall x\in [-1,x_0]: \quad u(x)>\alpha - S > 0 $.
	Neglecting the second term in $W(\, \cdot\,)$ we find
	\begin{align*}
	\frac{2}{\pi} W(u)&\ge \int^1_{-1} \frac{u(x) u''(x)^2}{(1+u'(x)^2)^{5/2}}\, dx
	>2(\alpha-S)  \int^{x_0}_{-1} \frac{ u''(x)^2}{(1+u'(x)^2)^{5/2}}\, dx\\
	&\ge 2(\alpha-S) \underbrace{\frac{1}{1+x_0}}_{\ge 1} \left(\int^{x_0}_{-1} \frac{ u''(x)}{(1+u'(x)^2)^{5/4}}\, dx\right)^2\\
	&\ge 2(\alpha-S)\, \left(\int^{u'(x_0)}_{u'(-1)=0} \frac{ 1}{(1+\tau^2)^{5/4}}\, d\tau\right)^2=2 (\alpha -S)G(S)^2
	\end{align*}
	and \eqref{eq:3.15} follows.
\end{proof}

This derivative estimate yields our second existence result for surfaces of revolution. 

\begin{proof}[Proof of Theorem \ref{theorem:1.4}]
Minimising sequences obey the bound in \eqref{eq:3.16} for some suitable $S\in (0,\alpha)$ and hence also in \eqref{eq:3.6} with $S$ in place  of $K$.
This shows that minimising sequences satisfy sufficiently strong $C^1$-a-priori-estimates. 
Proceeding similarly as in the proof of Theorem~\ref{theorem:1.3}, we  obtain the assertion in Theorem~\ref{theorem:1.4}. 
\end{proof}

\begin{remark}\label{theorem:3.5}
In order to interpret condition~\eqref{eq:1.6} one may observe that already $G(2.1)>1$ 
and that $\sup_{S\in (0,\infty)}G(S)^2=(c_0/2)^2=1.43\ldots$. 
This means that for $\alpha > 2.1$ the right hand side of \eqref{eq:1.6} is estimated from below 
by $(\alpha-2.1)\pi$, i.e. for $\alpha > 6.1$, condition~\eqref{eq:1.6} is weaker than condition~\eqref{eq:D}. 
For $\alpha \nearrow\infty$ one has asymptotically that $\max_{S\in[0,\alpha]}g_\alpha (S) =(c_0/2)^2\alpha =1.43\ldots \cdot \alpha$.
 
 The optimal $S$ for $g_\alpha(\,\cdot\,)$ cannot be calculated explicitly. However,
 analysing also $g_\alpha'(\,\cdot\,)$ shows that this is increasing in $\alpha$ and becomes unbounded for $\alpha \nearrow\infty$. See Figure~\ref{figure:8} for plots of  $g_\alpha(\,\cdot\,)$  for $\alpha = 6.1,\ 10,\ 50$.

 \begin{figure}[h] 
 	\hspace*{-1cm}
 	\includegraphics[width=.5\textwidth]{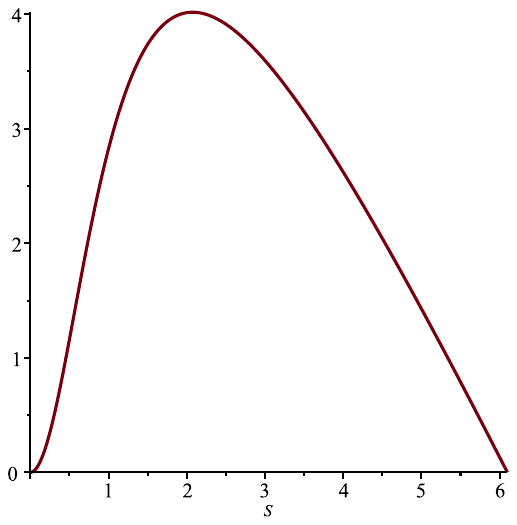}\hspace{-3cm}
 	\includegraphics[width=.5\textwidth]{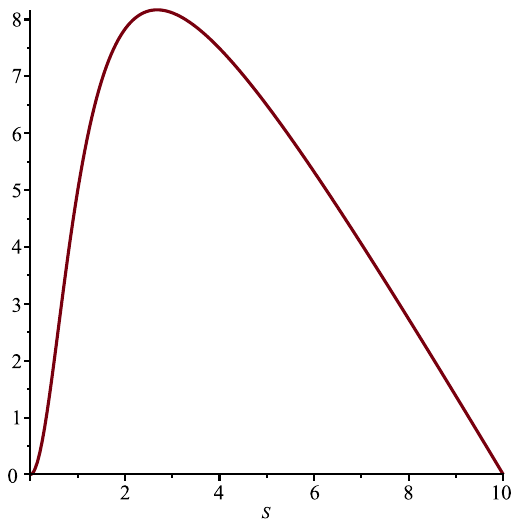}\hspace{-3cm}
 	\includegraphics[width=.5\textwidth]{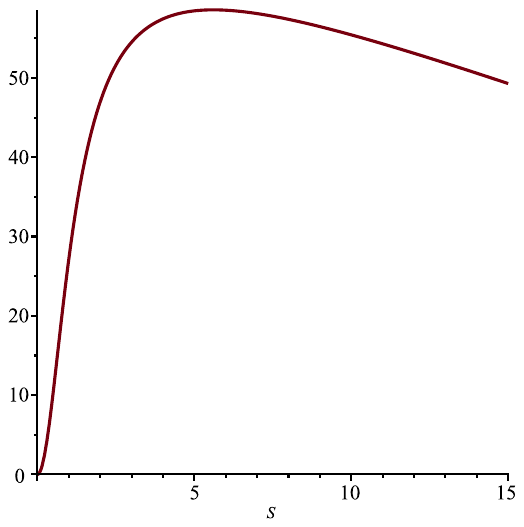}\hspace{-3cm}\vspace{-5cm}
 	\caption{Plots of $g_\alpha(\,\cdot\,)$ for $\alpha = 6.1,\ 10,\ 50$.}
 	\label{figure:8} 
 \end{figure} 
\end{remark}

\begin{remark}\label{theorem:3.6}
	How do obstacles look like that satisfy the assumptions of Theorem~\ref{theorem:1.4}?
	Since this result is of interest only for (relatively) large $\alpha$, we shall consider this case here only. 
	Then, as mentioned above, \eqref{eq:3.2} shows that the Willmore  functional resembles, up to the factor $u(x)$, the one-dimensional elatica functional. 
	For this reason we recall the functions $U_c:[-\frac{1}{2}, \frac{3}{2}]\to [-2/c_0,2/c_0]$ from Remark~\ref{theorem:2.4} and define for $\alpha >> 2$ and $0<c<c_0$:
	\begin{equation*}
		u_{\alpha,c}:[-1,1]\to [\alpha-4/c_0,\alpha],\quad 
		u_{\alpha,c}(x):=\alpha -U_c(x+1/2)+U_c(-1/2).
    \end{equation*}	
    According to Remark~\ref{theorem:2.4} the one-dimensional elastic energy of $U_c$ over $[-\frac{1}{2}, \frac{3}{2}]$ is $2c^2$. 
    In view of this we obtain from \eqref{eq:3.2}:
    \begin{align*}
    \dfrac{2}{\pi}W(u_{\alpha,c}) 
    &=  \int^1_{-1} \dfrac{u_{\alpha,c}''(x)^2u_{\alpha,c}(x)}{(1+(u_{\alpha,c}'(x))^2)^{5/2}}\, dx
    + \int^1_{-1} \dfrac{1}{u_{\alpha,c}(x) \sqrt{1+(u_{\alpha,c}'(x))^2}}\, dx\\
    &< \alpha \underbrace{ \int^1_{-1} \dfrac{u_{\alpha,c}''(x)^2}{(1+(u_{\alpha,c}'(x))^2)^{5/2}}\, dx}_{\mbox{the one-dimensional functional}}
    +2\frac{1}{\alpha-\frac{4}{c_0}}\\
    W(u_{\alpha,c}) &< \pi \alpha c^2+\frac{\pi}{\alpha-2}.
    \end{align*}
    In order that $u_{\alpha,c}$ obeys the condition in Theorem~\ref{theorem:1.4} we need to choose $c\in (0,c_0)$ such that 
    \begin{equation}\label{eq:3.17}
     \alpha c^2+\frac{1}{\alpha-2} \stackrel{!}{\le}
    \max_{S\in [0,\alpha]}\left( (\alpha-S)G(S)^2 \right)
    \end{equation}
	is satisfied. This condition cannot be resolved explicitly.
	For $\alpha\approx \infty$ the left hand side of \eqref{eq:3.17} behaves asymptotically like $\alpha c^2$ and the right hand side like $\alpha (c_0/2)^2$. 
	Hence, for $\alpha\approx \infty$, condition~\eqref{eq:3.17} is satisfied if $0<c<c_0/2\approx 1.198$. 
	The following table displays numerically calculated threshold values $c_{\mbox{\scriptsize thre}}$ such that for $0<c<c_{\mbox{\scriptsize thre}}$, 
	slightly enlarged  $u_{\alpha,c}$ yield admissible obstacles. 
	\begin{align*}
	\begin{array}{|c|c|c|} \hline 
	\alpha & \max_{S\in[0,\alpha]}g_\alpha (S) & c_{\mbox{\scriptsize thre}} \\ \hline
	10& 8.170\ldots & 0.896\ldots    \\ \hline
	25 &26.231\ldots &  1.023\ldots    \\ \hline
	50& 58.583\ldots  & 1.082\ldots  \\ \hline
	100& 125.756\ldots  & 1.121\ldots\\ \hline
	\end{array}
	\end{align*}
	In Figure~\ref{figure:9} we display how much the straight line $x\mapsto \alpha $
	may be pushed down, when $c=1$ and $c=c_0/2$ are admissible, respectively. One may observe that for
	$\alpha >25$ the admissible profiles change only a little; in particular $u_{\alpha,1}$
	is always admissible.

	\begin{figure}[h] 
		\hspace*{-4cm}
		\includegraphics[width=.8\textwidth]{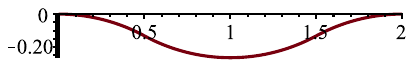}\hspace{-5cm}
		\raisebox{2.2cm}{\includegraphics[width=.65\textwidth]{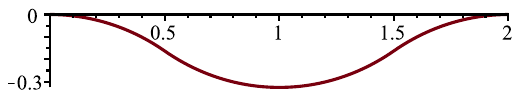}}\hspace{-6cm}{   }\vspace{-9cm}
		\caption{Plots of $x\mapsto -U_c(x+1/2)+U_c(-1/2)$ for $c=1$ (left) and $c=c_0/2$ (right).}
		\label{figure:9} 
	\end{figure} 
\end{remark}


\subsection{Regularity of minimisers}\label{subsection:3.3}
In order to show regularity we follow the strategy in Section \ref{subsection:2.1}. 
The formula for the first variation of  $W^h(u)$ is derived in \cite[formula (A.1)]{DDG} as follows: 
\begin{align}
\left(W^h\right)'(u)(\varphi) 
 =& 2 \int^1_{-1} \dfrac{\kappa_h(x)}{1 + u'(x)^2} \varphi''(x) \, dx 
   + \int^1_{-1} \kappa_h(x)^2 \dfrac{\sqrt{1 + u'(x)^2}}{u(x)^2} \varphi(x) \, dx \nonumber \\
 &  - 5 \int^1_{-1} \kappa_h(x)^2 \dfrac{u'(x)}{u(x) \sqrt{1 + u'(x)^2}} \varphi'(x) \, dx 
   - 2 \int^1_{-1} \dfrac{\kappa_h(x)}{u(x)^2} \varphi(x) \, dx\label{eq:3.18} \\
 &  + 4 \int^1_{-1} \kappa_h(x) \dfrac{u'(x)}{u(x) (1 + u'(x)^2)} \varphi'(x) \, dx  \nonumber  
\end{align}
for all $\varphi \in H^2_0(0,1)$. 
Let $u$ be a minimiser of $W^h(\cdot)$ in the admissible set $N_\alpha(\psi)$. 
Then $u$ satisfies the variational inequality: 
\begin{equation*}
\forall v \in N_\alpha(\psi):
\quad \left( W^h \right)'(u)(v-u) \ge 0. 
\end{equation*}
One may in particular choose
$v=u-\varphi$ with $\varphi \in H^2_0(0,1)$ and $\varphi\ge 0$ small enough because we push any
admissible function \emph{below} the obstacle. This means that
the minimiser $u$ is here a \emph{subsolution} of the hyperbolic elastica equation, i.e.
\begin{equation}
\label{eq:3.19}
\forall \varphi \in H^2_0(0,1):\quad \varphi\ge 0\quad \Rightarrow\quad 
\left( W^h \right)'(u)(\varphi) \le 0.
\end{equation}   
This shows that $-\left( W^h \right)'(u)$ is a nonnegative distribution
on $C^\infty_{\rm c}(-1,1)$ and hence even on $C^0_{\rm c}(-1,1)$, cf. Section~\ref{subsection:2.3} or \cite[Lemma 37.2]{Tartar}.
Hence, combining \eqref{eq:3.19} with the Riesz representation theorem, we find a nonnegative Radon measure $\mu$ such that 
\begin{equation}
\label{eq:3.20}
\left( W^h \right)'(u)(\varphi) =- \int^1_{-1} \varphi \, d\mu
\end{equation}
for all $\varphi \in C^\infty_{\rm c}(-1,1)$. 
We define $\mathcal{N} \subset (-1,1)$ by  
\begin{align*}
\mathcal{N} := \{ x \in (-1,1) \mid u(x) < \psi(x) \}. 
\end{align*}
Since $u$ and $\psi$ are continuous on $[-1,1]$, the set $\mathcal{N}$ is open, and we have: 
\begin{lemma} \label{theorem:3.7}
Assume that $\psi \in C^0([-1,1])$ satisfies condition \eqref{eq:C}.  
Suppose that there exists a minimiser $u \in N_\alpha(\psi)$  of $W^h(\cdot)$. 
Then there exist $ a \in (0,1)$ such that $(-1,-a) \cup (a,1) \subset \mathcal{N}$. 
\end{lemma}

\begin{lemma} \label{theorem:3.8}
Assume that $\psi \in C^0([-1,1])$ satisfies condition \eqref{eq:C}.  
Suppose that there exists a minimiser $u \in N_\alpha(\psi)$  of $W^h(\cdot)$. 
Then 
\begin{equation*}
\mu(-1,1) < \infty. 
\end{equation*}
\end{lemma}
\begin{proof}
	As in Section~\ref{subsection:2.3} one finds first that
	$$
	\mu(\mathcal{N})=0.
	$$
Hence, by Lemma \ref{theorem:3.7} we find $a\in (0,1)$ such that $\mu(-1,1) = \mu([-a,a])$. 
Fix $\eta \in C^\infty_{\rm c}(-1,1)$ with $\eta \equiv 1$ in $[-a,a]$ and $0 \le \eta \le 1$ in $(-1,1)$. 
It follows from $u \in N_\alpha(\psi)\subset H^2((-1,1);(0,\infty))\hookrightarrow C^1([-1,1],(0,\infty))$  that 
\begin{equation*}
|\kappa_h(x)| \le C(|u''(x)| + 1),  
\end{equation*}
and further  that 
\begin{align*}
\mu(-1,1) &=\mu([-a,a]) \le \int^1_{-1} \eta(x) d\mu 
  =- \left( W^h \right)'(u)(\eta) \\
 &\le C( \| u''\|_{L^2(-1,1)} \|\eta \|_{H^2(-1,1)} + \|\eta \|_{W^{2,1}(-1,1)} \\ 
 & \qquad \qquad +  \| u''\|_{L^2(-1,1)}^2 \| \eta \|_{W^{1, \infty}(-1,1)} ) \\
 &\le C \| \eta \|_{H^2(-1,1)}.  
\end{align*}
Therefore Lemma \ref{theorem:3.8} follows. 
\end{proof}

In order to study the regularity of minimisers, we employ the same ideas
already  used in \cite[Proof of Theorem~3.9, Step 2]{DDG}, see also Section \ref{subsection:2.3}. 
\begin{lemma} \label{theorem:3.9}
Fix $\eta \in C^\infty_{\rm c}(-1,1)$ and set 
\begin{align*}
\varphi_1(x) &:= \int^x_{-1} \! \int^y_{-1} \eta(s) \, ds dy + \alpha (x+1)^2 + \beta (x+1)^3, \\
\varphi_2(x) &:= \int^x_{-1} \eta(y) \, dy + \dfrac{1}{4}(-3(x+1)^2 +(x+1)^3) \int^1_{-1} \eta(y) \, dy, 
\end{align*}
for $x \in [-1,1]$, where 
\begin{equation*}
\alpha := \dfrac{1}{2}\int^1_{-1} \eta(y) \, dy - \dfrac{3}{4} \int^1_{-1} \! \int^y_{-1} \eta(s) \, ds dy, \qquad 
\beta := -\dfrac{1}{2}\alpha - \dfrac{1}{8}\int^1_{-1} \! \int^y_{-1} \eta(s) \, ds dy.
\end{equation*}
Then, $\varphi_1, \varphi_2 \in H^2_0(-1,1)$ and there exists $C>0$ such that 
\begin{gather*}
\| \varphi_1 \|_{C^1(-1,1)}, |\alpha|, |\beta| \le C \| \eta \|_{L^1(-1,1)}, \\
\| \varphi_2 \|_{L^\infty(-1,1)} \le C \| \eta\|_{L^1(-1,1)}, \qquad \| \varphi'_2 \|_{L^p(-1,1)} \le C \| \eta \|_{L^p(-1,1)} \quad \text{for} \quad p \in [1, \infty). 
\end{gather*}
\end{lemma}

\begin{proposition} \label{theorem:3.10}
Assume that $\psi $ satisfies condition~\eqref{eq:C}.
Suppose that there exists a minimiser $u \in N_\alpha(\psi)$  of $W^h(\cdot)$.   
Then $u\in C^2([-1,1])$, $u''$ is weakly differentiable and $u'''\in BV(-1,1)$. 
On the complement of the coincidence set the minimiser is even smooth, i.e. $u|_{\mathcal N}\in C^\infty ({\mathcal N})$.
\end{proposition}
\begin{proof}
We define a nonnegative bounded increasing function $m : (-1,1) \to \mathbb{R}$ by 
\begin{equation*}
m(x) = \mu(-1,x) \quad \text{for} \quad x \in (-1,1). 
\end{equation*}
Then, along the same lines as in the proof of Proposition \ref{theorem:2.8}, we obtain 
\begin{equation} \label{eq:3.21}
\int^1_{-1} \varphi \, d \mu(x) = - \int ^1_{-1} m(x) \varphi'(x) \, dx 
\end{equation}
for all $\varphi \in C^\infty_{\rm c}(-1,1)$. 
It follows from \eqref{eq:3.18}, \eqref{eq:3.20} and \eqref{eq:3.21} that 
\begin{equation}
\label{eq:3.22}
\begin{aligned}
&2 \int^1_{-1} \dfrac{\kappa_h(x)}{1 + u'(x)^2} \varphi''(x) \, dx \\
& \quad = - \int^1_{-1} \dfrac{\kappa_h(x)^2 \sqrt{1 + u'(x)^2}}{u(x)^2} \varphi(x) \, dx \\
& \qquad + 5 \int^1_{-1} \dfrac{\kappa_h(x)^2 u'(x)}{u(x) \sqrt{1 + u'(x)^2}} \varphi'(x) \, dx   
               + 2 \int^1_{-1} \dfrac{\kappa_h(x)}{u(x)^2} \varphi(x) \, dx \\
 & \qquad - 4 \int^1_{-1} \dfrac{\kappa_h(x) u'(x)}{u(x) (1 + u'(x)^2)} \varphi'(x) \, dx
  + \int ^1_{-1} m(x) \varphi'(x) \, dx  
\end{aligned}
\end{equation}
for all $\varphi \in C^\infty_{\rm c}(-1,1)$. 
By a density argument, we see that \eqref{eq:3.22} also holds for all $\varphi \in H^2_0(-1,1)$. 

Fix $\eta \in C^\infty_{\rm c}(-1,1)$ arbitrarily. 
Taking $\varphi_1$ as $\varphi$ in \eqref{eq:3.22}, where $\varphi_1$ is defined in Lemma~\ref{theorem:3.9}, we have 
\begin{align}
& 2 \int^1_{-1} \dfrac{\kappa_h(x)}{1 + u'(x)^2} \eta(x) \, dx \nonumber\\
& \quad = -4 \int^1_{-1} \dfrac{\kappa_h(x)}{1 + u'(x)^2} (\alpha + 3 \beta x + 3\beta) \, dx
 - \int^1_{-1} \dfrac{\kappa_h(x)^2 \sqrt{1 + u'(x)^2}}{u(x)^2} \varphi_1(x) \, dx \nonumber\\
& \qquad + 5 \int^1_{-1} \dfrac{\kappa_h(x)^2 u'(x)}{u(x) \sqrt{1 + u'(x)^2}} \varphi'_1(x) \, dx   
               + 2 \int^1_{-1} \dfrac{\kappa_h(x)}{u(x)^2} \varphi_1(x) \, dx \label{eq:3.23}\\
 & \qquad - 4 \int^1_{-1} \dfrac{\kappa_h(x) u'(x)}{u(x) (1 + u'(x)^2)} \varphi'_1(x) \, dx 
  + \int ^1_{-1} m(x) \varphi'_1(x) \, dx \nonumber\\
& \quad =: J_1 + J_2 + J_3 + J_4 + J_5+J_6. \nonumber
\end{align}
Because any minimiser belongs to $N_\alpha (\psi)\subset C^1([-1,1],(0,\infty))$, we observe from Lemmas  \ref{theorem:3.8} and \ref{theorem:3.9} that  
\begin{align*}
|J_1| & \le C(|\alpha| + |\beta|) (\| u'' \|_{L^1(-1,1)}+1) \le C \| \eta\|_{L^1(-1,1)}, \\
|J_2|+|J_3| & \le C (\| u'' \|^2_{L^2(-1,1)}+1) \|\varphi_1 \|_{C^1([-1,1])} \le C \| \eta \|_{L^1(-1,1)}, \\
|J_4| & \le C(\| u'' \|_{L^1(-1,1)}+1) \| \varphi_1 \|_{C^0([-1,1])} \le C \| \eta \|_{L^1(-1,1)}, \\
|J_5| & \le C(\| u'' \|_{L^1(-1,1)}+1) \| \varphi_1 \|_{C^1([-1,1])} \le C \| \eta \|_{L^1(-1,1)}, \\
|J_6| & \le \Bigl(\sup_{x \in [-1,1]} m(x)\Bigr)  \| \varphi_1 \|_{C^1([-1,1])} \le C \mu(-1,1) \| \eta \|_{L^1(-1,1)} \le C \| \eta \|_{L^1(-1,1)}.  
\end{align*}
This together with \eqref{eq:3.23} implies that 
\begin{equation}
\label{eq:3.24}
\| \kappa_h(x) (1+u'(x)^2)^{-1} \|_{L^\infty(-1,1)} \le C. 
\end{equation}
Combining \eqref{eq:3.24} with $u \in C^1([-1,1],(0,\infty))$, we obtain 
\begin{equation}
\label{eq:3.25}
\| u'' \|_{L^\infty(-1,1)} \le C. 
\end{equation}
This shows that we already have  $u\in W^{2,\infty}((-1,1);(0,\infty))$.

Fix $\eta \in C^\infty_{\rm c}(-1,1)$ arbitrarily. 
Taking $\varphi_2$ as $\varphi$ in \eqref{eq:3.22}, where $\varphi_2$ is defined by Lemma~\ref{theorem:3.9}, we have 
\begin{align*}
& 2 \int^1_{-1} \dfrac{\kappa_h(x)}{1 + u'(x)^2} \eta'(x) \, dx 
 = -3
 \left( \int^1_{-1} x \dfrac{\kappa_h(x)}{1 + u'(x)^2} \, dx\right) \ \left(\int^1_{-1} \eta(x) \, dx\right)  \\
 & \qquad \qquad  - \int^1_{-1} \dfrac{\kappa_h(x)^2 \sqrt{1 + u'(x)^2}}{u(x)^2} \varphi_2(x) \, dx\\
& \qquad \qquad  + 5 \int^1_{-1} \dfrac{\kappa_h(x)^2 u'(x)}{u(x) \sqrt{1 + u'(x)^2}} \varphi'_2(x) \, dx  
          + 2 \int^1_{-1} \dfrac{\kappa_h(x)}{u(x)^2} \varphi_2(x) \, dx \\
& \qquad \qquad - 4 \int^1_{-1} \dfrac{\kappa_h(x) u'(x)}{u(x) (1 + u'(x)^2)} \varphi'_2(x) \, dx 
  + \int ^1_{-1} m(x) \varphi'_2(x) \, dx \\
& =: J'_1 + J'_2 + J'_3 + J'_4 + J'_5+ J'_6. 
\end{align*}
We deduce from \eqref{eq:3.25} and $u \in  H^2((-1,1);(0,\infty))\hookrightarrow C^1([-1,1],(0,\infty))$
 that 
\begin{align*}
|J'_2|+|J'_3| \le C (\| u''\|_{L^\infty(-1,1)}^2 +1) \| \varphi_2\|_{W^{1,1}(-1,1)} 
       \le C \| \eta \|_{L^1(-1,1)}. 
\end{align*}
Similarly as above, we obtain 
\begin{align*}
|J'_1| &\le C (\| u''\|_{L^\infty(-1,1)}+1) \| \eta\|_{L^1(-1,1)} \le C \| \eta\|_{L^1(-1,1)}, \\
|J'_4| &\le C (\| u''\|_{L^\infty(-1,1)}+1) \| \varphi_2 \|_{L^1(-1,1)} \le C \| \eta\|_{L^1(-1,1)}, \\
|J'_5| &\le C (\| u''\|_{L^\infty(-1,1)}^2 +1) \| \varphi'_2\|_{L^1(-1,1)} \le C \| \eta \|_{L^1(-1,1)}, \\
|J'_6| &\le \Bigl( \sup_{x \in [-1,1]} m(x) \Bigr) \| \varphi'_2\|_{L^1(-1,1)} \le C \mu(-1,1) \|\eta\|_{L^1(-1,1)} \le C \|\eta\|_{L^1(-1,1)}.  
\end{align*}
Thus we observe that 
\begin{equation*}
\Bigl| \int^1_{-1} \dfrac{\kappa_h(x)}{1+u'(x)^2} \eta'(x) \, dx \Bigr| \le C \| \eta\|_{L^1(-1,1)} 
\end{equation*}
for all $\eta \in C^\infty_{\rm c}(-1,1)$, and then 
\begin{equation*}
\| (\kappa_h (1 + (u')^2)^{-1})' \|_{L^{\infty}(-1,1)} \le C. 
\end{equation*}
This together with $u\in W^{2,\infty}((-1,1);(0,\infty))$ (see \eqref{eq:3.25})  implies that 
\begin{equation*}
\| u'''\|_{L^{\infty}(-1,1)} <\infty. 
\end{equation*}
Finally, with the help of the absolutely continuous function
$$
f(x):=  \int^x_{-1} \dfrac{\kappa_h(\xi)^2 \sqrt{1 + u'(\xi)^2}}{u(\xi)^2} \, d\xi    
- 2 \int^x_{-1} \dfrac{\kappa_h(\xi)}{u(\xi)^2}  \, d\xi,
$$
\eqref{eq:3.22} may be written in the following form
\begin{align*}
&2 \int^1_{-1} \left( \dfrac{\kappa_h(x)}{1 + u'(x)^2}\right)' \varphi'(x) \, dx 
+\int^1_{-1} f(x) \varphi'(x)\, dx  \\
& \qquad
 + 5 \int^1_{-1} \dfrac{\kappa_h(x)^2 u'(x)}{u(x) \sqrt{1 + u'(x)^2}} \varphi'(x) \, dx   
 - 4 \int^1_{-1} \dfrac{\kappa_h(x) u'(x)}{u(x) (1 + u'(x)^2)} \varphi'(x) \, dx \\
 & \qquad
+ \int ^1_{-1} m(x) \varphi'(x) \, dx  =0
\end{align*}
for all $\varphi \in C^\infty_{\rm c}(-1,1)$. This shows that there exists
a constant $c\in\mathbb{R}$ such that
\begin{equation*}
2 \left( \dfrac{\kappa_h(x)}{1 + u'(x)^2}\right)' 
+ f(x) 
+ 5  \dfrac{\kappa_h(x)^2 u'(x)}{u(x) \sqrt{1 + u'(x)^2}}  
- 4  \dfrac{\kappa_h(x) u'(x)}{u(x) (1 + u'(x)^2)} 
+  m(x) =c.
\end{equation*}
Since $m$ is increasing and bounded and we already know that $u\in W^{3,\infty}(-1,1)$ we conclude that $\kappa_h'$ and also $u'''$ is of bounded
variation. Since $u|_{\mathcal N}$ is a solution of the elliptic differential  equation \eqref{eq:3.19} for hyperbolic elasticae, further boot strapping as in Step 2 of the proof  of \cite[Theorem 3.9]{DDG} shows finally that $u|_{\mathcal N}\in C^\infty ({\mathcal N})$.
Proposition~\ref{theorem:3.10} follows. 
\end{proof}

\noindent
{\bf Acknowledgements.}
This work was initiated during the first author's visit at Tohoku University
(before the Corona pandemic).
The first author is very grateful to the second author for his warm hospitality and the inspiring working atmosphere.
The second author was supported in part by JSPS KAKENHI Grant Numbers JP19H05599 and JP20KK0057.

Both authors are very grateful to the referees for their careful reading of the manuscript and for numerous very helpful remarks, questions, and suggestions.



\begin{thebibliography}{99}
	
\bibitem{BryantGriffiths_1986}
R. Bryant and P.  Griffiths, {\it Reduction for constrained variational problems and $\int\frac{1}{2}k\sp 2\, ds$},
Amer. J. Math. {\bf 108} (1986), no. 3, 525--570.

\bibitem{dLPR} 
F.~Da~Lio, F. Palmurella and T. Rivi\`{e}re,
{\it A resolution of the Poisson problem for elastic plates},
 Arch. Ration. Mech. Anal. {\bf 236} (2020), no. 3, 1593--1676. 

\bibitem{DD}
A. Dall'Acqua and K. Deckelnick, {\it An obstacle problem for elastic graphs}, 
SIAM J. Math. Anal. {\bf 50}  (2018), no. 1, 119--137. 

\bibitem{DDG}
A. Dall'Acqua, K. Deckelnick and H.-Ch. Grunau, {\it Classical solutions to the Dirichlet problem for Willmore surfaces of revolution}, 
Adv. Calc. Var. {\bf 1} (2008), no. 4, 379--397. 

\bibitem{DFGS}
A. Dall'Acqua, S. Fr\"ohlich, H.-Ch. Grunau and F. Schieweck, 
{\it Symmetric Willmore surfaces of revolution satisfying arbitrary Dirichlet boundary data}, 
Adv. Calc. Var. {\bf 4} (2011), no. 1, 1--81. 

\bibitem{DG07}
K. Deckelnick and H.-Ch. Grunau, {\it Boundary value problems for the one-dimensional Willmore equation}, 
Calc. Var. Partial Differential Equations {\bf 30} (2007), no. 3, 293--314.

\bibitem{DGR}
K. Deckelnick, H.-Ch. Grunau and  M. R\"oger, {\it Minimising a relaxed Willmore functional for graphs subject to boundary conditions},
Interfaces Free Bound. {\bf 19} (2017), no. 1,  109--140.

\bibitem{Eichmann2016}
S. Eichmann, 
{\it Nonuniqueness for Willmore surfaces of revolution satisfying Dirichlet boundary data},
J. Geom. Anal. {\bf 26} (2016), no. 4, 2563--2590.

\bibitem{Eichmann2019}
S. Eichmann, {\it The Helfrich boundary value problem},
{ Calc. Var. Partial Differential Equations} {\bf 58} (2019), no. 1, Paper No. 34, 26 pp. 

\bibitem{EichGr}
S.~Eichmann and H.-Ch. Grunau,
{\it Existence for Willmore surfaces of revolution satisfying non-symmetric Dirichlet boundary conditions},
{  Adv. Calc. Var.} {\bf 12} (2019), no. 4, 333--361. 
 
\bibitem{EK} S. Eichmann and A. Koeller, 
{\it Symmetry for Willmore surfaces of revolution}, 
J. Geom. Anal. {\bf 27}  (2017), no. 1, 618--642. 

\bibitem{Euler} L. Euler, ``Opera Omnia",  Ser. 1, {\bf 24}, Z\"urich: Orell F\"ussli, 1952.

\bibitem{GGS}
F. Gazzola, H.-Ch. Grunau and G. Sweers, ``Polyharmonic Boundary Value Problems", 
Lecture Notes in Mathematics {\bf 1991}, Springer-Verlag, Berlin, 2010. 

\bibitem{G}
H.-Ch. Grunau, {\it The asymptotic shape of a boundary layer of symmetric Willmore surfaces of revolution},
In: Inequalities and applications 2010, 
Internat. Ser. Numer. Math. {\bf 161}, 19--29, Birkh\"auser/Springer, Basel, 2012. 

\bibitem{Herterich_Pinkall}
U. Hertrich-Jeromin and U. Pinkall, {\it  Ein Beweis der Willmoreschen Vermutung f\"ur Kanaltori},
J. Reine Angew. Math. {\bf 430} (1992), 21--34.

\bibitem{KuSch}
E. Kuwert and R. Sch\"atzle, 
{\it The Willmore functional}, In: Topics in modern regularity theory, Ed. Norm., Pisa, CRM Series {\bf 13} (2012), 1--115. 

\bibitem{LangerSinger_1984a}
J. Langer and D. Singer, {\it The total squared curvature of closed curves},
J. Differ. Geom. {\bf 20} (1984), no. 1, 1--22.

\bibitem{LangerSinger_1984b}
J. Langer and D. Singer, {\it  Curves in the hyperbolic plane and mean curvature of tori in $3$-space},
Bull. London Math. Soc. {\bf 16} (1984), no. 5, 531--534. 
 
\bibitem{Mandel1} R. Mandel, {\it Boundary value problems for Willmore curves in $\mathbb{R}^2$}, 
Calc. Var. Partial Differential Equations {\bf 54} (2015), no. 4, 3905--3925.

 \bibitem{Mandel2} R. Mandel, {\it Explicit formulas, symmetry and symmetry breaking for Willmore surfaces of revolution}, 
 Ann. Global Anal. Geom. {\bf 54} (2018), no. 2, 187--236. 
 
\bibitem{MarquesNeves}
F.C. Marques and A.~Neves,
{\it The {W}illmore {C}onjecture},
Jahresber. Dtsch. Math.-Ver. {\bf 116} (2014), no. 4, 201--222.

\bibitem{Miura}
T. Miura, {\it Polar tangential angles and free elasticae}, Math. Eng. {\bf 3} (2021), no. 4, Paper No. 034, 12pp. 

\bibitem{Mueller_2019}
M. M\"uller, 
{\it An obstacle problem for elastic curves: existence results}, Interfaces Free Bound. {\bf 21} (2019), no. 1, 87--129. 

\bibitem{Mueller_2020}
M. M\"uller, 
{\it On gradient flows with obstacle and Euler's elstica}, Nonlinear Anal. {\bf 192} (2020), Paper No. 111676, 48pp. 

\bibitem{Mueller_2021}
M. M\"uller, 
{\it The elastic flow with obstacle: small obstacle results}, 
Appl. Math. Optim. {\bf 84} (2021), suppl. 1, S355--S402. 

\bibitem{NovagaPozzetta}
M. Novaga and M. Pozzetta, 
{\it Connected surfaces with boundary minimizing the Willmore energy},
Math. Eng. {\bf  2} (2020), no. 3, 527--556.

\bibitem{Okabe-Yoshizawa} S. Okabe and K. Yoshizawa, 
{\it A dynamical approach to the variational inequality on modified elastic graphs}, Geom. Flows {\bf 5} (2020), no. 1, 78--101. 

\bibitem{Poisson}
S.D. Poisson,
{\it M\'{e}moire sur les surfaces \'{E}lastiques},
{M\'{e}m. de l'Inst.} (1812; pub. 1816), 167--226.

\bibitem{Pozzetta}
M. Pozzetta, {\it On the Plateau-Douglas Problem for the Willmore energy of surfaces with planar boundary curves}, 
ESAIM Control Optim. Calc. Var., Paper No. S2, 35 pp.

\bibitem{Saks} S. Saks, ``Theory of the integral", Second revised edition, English translated by L. C. Young, 
With two additional notes by Stefan Banach, Dover Publications, Inc., New York, 1964. 

\bibitem{Schaetzle}
R.~Sch{\"a}tzle, {\it The {W}illmore boundary problem}, 
{Calc. Var. Partial Differential Equations} {\bf  37} (2010), no. 3-4, 275--302.

\bibitem{Tartar} L. Tartar, ``An introduction to Sobolev spaces and interpolation spaces", 
Lecture Notes of the Unione Matematica Italiana {\bf 3}, Springer-Verlag  Berlin, UMI Bologna, 2007.

\bibitem{Wil65} T.J. Willmore,
{\it Note on embedded surfaces},
{An. \c{S}tiin\c{t}. Univ. Al. I. Cuza Ia\c{s}i
	Se\c{c}t. I a Mat} {\bf 11} (1965), 493--496.

\bibitem{Yoshizawa}
K. Yoshizawa, {\it A remark on elastic graphs with the symmetric cone obstacle}, SIAM J. Math. Anal.  {\bf 53} (2021), no. 2, 1857--1885. 

\end{thebibliography}
\end{document}